\documentclass[reqno]{amsart}

\usepackage{etex}
\usepackage{amsmath,amssymb,amsthm,amsfonts,mathrsfs}
\usepackage[frame,cmtip,arrow,matrix,line,graph,curve]{xy}
\usepackage{graphpap,color,paralist,pstricks}
\usepackage[mathscr]{eucal}
\usepackage[matha,mathb]{mathabx}
\usepackage[pdftex,colorlinks,backref=page,citecolor=blue]{hyperref}
\usepackage{tikz}
\usetikzlibrary{decorations,calc,matrix,cd}
\usepackage{epic,eepic}
\usepackage{yfonts}
\usepackage{enumerate}
\usepackage{blkarray}
\usepackage{mathtools}
\usepackage{setspace}

\usepackage{xcolor}

\renewcommand{\green}[1]{\textcolor{darkgreen}{#1}}
\definecolor{darkgreen}{rgb}{0,.5,0}
\definecolor{brown}{rgb}{0.5,0.3,0}

\usepackage[T1]{fontenc}
\usepackage[utf8]{inputenc}

\theoremstyle{definition}
\newtheorem{theorem}{Theorem}[section]
\newtheorem{definition}[theorem]{Definition}

\newtheorem{lemma}[theorem]{Lemma}
\newtheorem{proposition}[theorem]{Proposition}
\newtheorem{corollary}[theorem]{Corollary}

\newtheorem{propdef}[theorem]{Proposition/Definition}
\theoremstyle{remark}
\newtheorem{remark}[theorem]{Remark}
\newtheorem{example}[theorem]{Example}

\makeatletter
\newtheorem*{rep@theorem}{\rep@title} \newcommand{\newreptheorem}[2]{%
\newenvironment{rep#1}[1]{%
\def\rep@title{\bf #2 \ref{##1}}%
\begin{rep@theorem} }%
{\end{rep@theorem} } }
\makeatother
\newreptheorem{theorem}{Theorem}
\newreptheorem{conjecture}{Conjecture}

\numberwithin{equation}{section}

\newcommand{\set}[1]{\left\{#1\right\}}
\newcommand{\abs}[1]{\left|#1\right|}

\DeclareMathOperator{\D}{{\mathscr D}}
\DeclareMathOperator{\BC}{BC}
\DeclareMathOperator{\RBC}{\overline{BC}}
\DeclareMathOperator{\IN}{{IN}}

\DeclareMathOperator{\M}{\mathrm M}         
\DeclareMathOperator{\cl}{cl}               

\DeclareMathOperator{\Tr}{{Tr}}  
\DeclareMathOperator{\nbc}{{\mathit{nbc}}}  
\DeclareMathOperator{\EA}{{EA}}  
\DeclareMathOperator{\EP}{{EP}}  
\DeclareMathOperator{\IA}{{IA}}  
\DeclareMathOperator{\IP}{{IP}}  

\DeclareMathOperator{\rank}{{rk}}

\DeclareMathOperator{\Des}{{Des}}

\renewcommand{\P}{{\mathscr P}}
\newcommand{\Q}{{\mathscr Q}}

\newcommand{\B}{{\mathscr B}}
\newcommand{\F}{{\mathscr F}}
\newcommand{\G}{{\mathscr G}}

\newcommand{\T}{{\mathscr T}}

\newcommand{\ee}{\mathcal{E}}
\renewcommand{\j}{\,\mathsf{J}}

\newcommand{\AAA}{{A_{\M,\M^\perp}}}  
\newcommand{\SSS}{{S_{\M,\M^\perp}}}  
\newcommand{\III}{{I_{\M,\M^\perp}}}  
\newcommand{\JJJ}{{J_{\M,\M^\perp}}}  
\newcommand{\og}{\overline{\gamma}}

\DeclareMathOperator{\Asc}{{Asc}} 

\begin{document}

\title{Lagrangian combinatorics of matroids}

\author{Federico Ardila}
\address{San Francisco State University and Universidad de Los Andes}
\email{federico@sfsu.edu}

\author{Graham Denham}
\address{University of Western Ontario}
\email{gdenham@uwo.ca}

\author{June Huh}
\address{Princeton University}
\email{huh@princeton.edu}

\maketitle

\begin{abstract}
The \emph{Lagrangian geometry of matroids} was introduced in \cite{ADH} through the construction of the \emph{conormal fan} of a matroid $\M$. We used the conormal fan to give a Lagrangian-geometric
interpretation of 
the $h$-vector of the broken circuit complex of $\M$: 
its entries are the degrees of the mixed intersections of  
  certain convex piecewise linear functions $\gamma$ and $\delta$ on the conormal fan of  $\M$.
By showing that the conormal fan satisfies the Hodge-Riemann relations, we proved Brylawski's conjecture that this $h$-vector is a log-concave sequence.

This sequel explores the \emph{Lagrangian combinatorics of matroids}, further developing the combinatorics of biflats and
biflags of a matroid, and relating them to the theory of basis activities developed by Tutte, Crapo, and Las Vergnas. Our main result is a combinatorial strengthening of the $h$-vector computation: we write  the $k$-th mixed intersection of $\gamma$ and $\delta$ 
explicitly as a sum of biflags corresponding to the $\nbc$-bases  of internal activity $k+1$.
\end{abstract}

\tableofcontents

\section{Introduction.} \label{sec:intro}

Let $\M$ be a matroid of rank $r+1$ on $n+1$ elements with no loops and no coloops.
The \emph{Lagrangian geometry of matroids} was introduced in
\cite{ADH} through the construction of the \emph{conormal fan}  of  $\M$. The conormal fan of $\M$ 
is a Lagrangian analogue of the Bergman fan of $\M$, which
in turn is a tropical geometric model of $\M$. We used the conormal fan
to give a tropical geometric interpretation of 
the $h$-vector
of the broken circuit complex $\BC(\M)$. 
Explicitly, we identified convex piecewise linear functions $\gamma$ and $\delta$  such that 
\[
\gamma^k \delta^{n-k-1} \cap 1_{\M,\M^\perp}
 = 
h_{r-k}(\BC(\M)) \ \ \text{for all $k$,} 
\]
where 
$1_{\M,\M^\perp}$  is the top-dimensional constant Minkowski weight $1$ on the conormal fan of $\M$ \cite[Theorem 1.2]{ADH}. 
We also showed that the conormal fan is \emph{Lefschetz}, and in particular,  satisfies the Hodge--Riemann
relations \cite[Theorem 5.27]{ADH}. 
Combining these results, we proved Brylawski's conjecture from \cite{Brylawski82} that 
the $h$-vector
of the  broken circuit complex
forms a log-concave sequence,
that is,
\[
h_i(\BC(\M))^2 \geq h_{i-1}(\BC(\M))h_{i+1}(\BC(\M)) \ \  \text{for all $i$.}
\]

In this followup paper, we explore the \emph{Lagrangian combinatorics
  of matroids}, which studies the combinatorial structure of
the conormal fan. We further develop the study of biflats and biflags
of matroids, initiated in \cite{ADH}, unveiling a strong connection
to the theory of basis activities developed by Tutte \cite{Tutte67},
Crapo \cite{Cr69}, and Las Vergnas
\cite{LasVergnas13}.
In particular, we directly relate the mixed intersections $\gamma^k \delta^{n-k-1}$ to the reduced broken circuit complex using the \emph{canonical expansion} in Section \ref{sec:canonical},
obtaining  the following bijective strengthening of \cite[Theorem 1.2]{ADH} in the \emph{conormal Chow ring} of Definition \ref{def:Chowring}.

\begin{theorem}\label{thm:main} 
For $0 \le k \le r$,
in the conormal Chow ring of $\M$, we have
\[
\gamma^k \delta^{n-k-1} = \sum_B 
x_{\F^+(B)| \G^+(B)},
\]
 where 
the sum is over the $\nbc$-bases of $\M$ of internal activity $k+1$.
\end{theorem}

The symbol $\F^+(B)| \G^+(B)$ stands for the extended  $\nbc$-biflag associated to $B$  in Section \ref{sec:nbc cones}. 
By construction, for every $B$, 
\[
 x_{\F^+(B)| \G^+(B)} \cap 1_{\M,\M^\perp} =1.
\]
Since $\M$ has exactly $h_{r-k}(\BC(\M))$ bases with internal activity $k+1$,
Theorem \ref{thm:main} implies \cite[Theorem 1.2]{ADH}. 
The original proof of  \cite[Theorem 1.2]{ADH} 
relied on the special case $k=0$ of Theorem \ref{thm:main}  \cite[Proposition 4.9]{ADH} and 
on the theory of Chern--Schartz--MacPherson cycles of matroids
introduced by L\'opez de Medrano, Rinc\'on, and Shaw \cite{LRS17}.

\subsection{Enumerative combinatorics of matroids.}\label{sec:enum}

For the remainder of this paper, we fix a total ordering on the ground set $E$ of $\M$ and identify $E$ with the set $\{0,1,\ldots,n\}$.
 We refer to  \cite{Welsh} and \cite{Oxbook} for any undefined matroid terminology.
We are interested in the following $r$-dimensional simplicial complexes associated to $\M$:
\begin{enumerate}[$\bullet$]\itemsep 5pt
\item The \emph{independence complex} $\IN(\M)$, the collection of subsets of $E$ which do not contain any circuit of $\M$.
\item The \emph{broken circuit complex} $\BC(\M)$, the collection of subsets of $E$ which do not contain any broken circuit of $\M$.
\end{enumerate}
A \emph{broken circuit} is a subset obtained from a circuit of $\M$ by deleting the least element in the fixed ordering on $E$. 
The broken circuit complex $\BC(\M)$ is the cone over the \emph{reduced broken circuit complex} $\RBC(\M)$ with apex $0$. 

For a simplicial complex $\Delta$ of dimension $r$, its \emph{$f$-vector}
$f(\Delta) = (f_0, f_1,\ldots, f_{r+1})$ is defined by
\[
f_i(\Delta) = \text{the number of faces in $\Delta$ with $i$ vertices}.
\]
The $f$-vector is often stored more compactly in the \emph{$h$-vector}
$h(\Delta) = (h_0, h_1,\ldots, h_{r+1})$, given by
\[
\sum_{i=0}^{r+1} f_i(\Delta) q^{r-i+1}=\sum_{i=0}^{r+1} h_i(\Delta) (q+1)^{r-i+1}.
\]
The $h$-vector of the broken circuit complex is given by
\[
h_{r-k}(\BC(\M)) = h_{r-k}(\RBC(\M)) = t_{k+1,0}(\M) \ \ \text{for all $k$}, 
\]
where $t_{i,j}(\M)$ is the coefficient of $x^iy^j$ in the Tutte polynomial  $T_{\M}(x,y)$ \cite{Bjorner}. 
In particular, $h_{r+1}(\BC(\M))$ is zero, and $h_r(\BC(\M))$ is  Crapo's \emph{beta invariant}
\[
\beta_{\M} 
= t_{1,0}(\M).
\]
%
The authors of \cite{AHK} and \cite{ADH} proved the following results,
conjectured by Mason and Hoggar 
\cite{Mason72, Hoggar74} and by Brylawski and Dawson \cite{Brylawski82, Dawson83}, respectively.

\begin{theorem} \label{th:flogconcave}
The following hold for any matroid $\M$.
\begin{enumerate}[(1)]\itemsep 5pt
\item
The $f$-vectors of $\IN(\M)$ and $\BC(\M)$ are log-concave \cite{AHK}.
\item The $h$-vectors of $\IN(\M)$ and $\BC(\M)$ are log-concave \cite{ADH}.
\end{enumerate}
\end{theorem}

We note that the independence complex of any matroid is the reduced broken circuit complex of another matroid \cite[Theorem 4.2]{Brylawski77}, and,
 for any simplicial complex, the log-concavity of its $h$-vector implies the log-concavity of its $f$-vector \cite[Corollary 8.4]{Brenti}.
Also, for  Theorem \ref{th:flogconcave}, we may suppose that $\M$ has no loops and no coloops.
We thus focus on the $h$-vector of the broken circuit complex of a matroid with no loops and no coloops.
One of the main ingredients in the proof of Theorem \ref{th:flogconcave} in this case is the above-mentioned formula for the $h$-vector \cite[Theorem 1.2]{ADH}, which we strengthen in Theorem \ref{thm:main}.

\begin{remark}
For matroids representable over the field of complex numbers, the intersection theoretic formula for the  $h$-vector of the broken circuit complex was given in
\cite{DGS}, and
the connection to the Chern--Schwartz--MacPherson classes was observed in \cite{HuhML} and \cite{Huh15}.
Varchenko's conjecture  on the number of critical points of products of linear forms \cite{Varchenko},
proved by Orlik and Terao in \cite{OT95},
is equivalent to the central special case of the formula
\[
\delta^{n-1} \cap 1_{\M,\M^\perp}=h_r(\BC(\M)).
\]
Recently, for any matroid, Berget, Eur, Spink, and Tseng proved a very general and closely related formula 
\[
\sum_{i+j+k+l=n} \Bigg(\int_{X_E} \alpha^i  \beta^j  c_k(\mathscr{S}^{\vee}_{\M}) c_l(\mathscr{Q}_{\M})\Bigg) (x+y)x^iy^jz^kw^l= (y+z)^r (x+w)^{n-r+1} T_{\M}\Big(\frac{x+y}{y+z},\frac{x+y}{x+w}\Big),
\]
where $c_k(\mathscr{S}^{\vee}_{\M})$ and $c_l(\mathscr{Q}_{\M})$ are \emph{tautological Chern classes} of $\M$ \cite[Theorem A]{BEST}, and they
used it to give another proof of Theorem \ref{th:flogconcave}.
To deduce Theorem \ref{th:flogconcave} from their formula, 
they used the fact \cite{AFR, DF} that many functions of matroids behave valuatively under matroid polytope subdivisions. This allowed them to reduce key computations to the case of representable matroids and prove them using an algebro-geometric argument, thus avoiding the combinatorics of biflats and biflags.
The final part of their proof also 
employs the fact, proved in \cite{ADH}, 
 that the conormal fan of a matroid is a Lefschetz fan. 
We refer to \cite[Remark 9.9]{BEST} for a detailed comparison of the two proofs of Theorem \ref{th:flogconcave}.
\end{remark}


\begin{example}\label{ex:graphs}
We will use two running examples throughout the paper. 
The first is the graphical matroid
 of the graph $G$ of the pyramid, whose dual  is also the matroid of the pyramid $G^\perp$. The
second is the graphical matroid  of the graph $H$ of the cube,
whose dual is the graphical matroid of the graph
$H^\perp$ of the octahedron. These are shown in Figure
\ref{fig:pyramid} and Figure \ref{fig:cube}. The $f$-vectors and
$h$-vectors of their  broken circuit complexes are shown in Table \ref{table:fh}.\end{example}

\begin{table}[h]
\begin{center}
\begin{tabular}{|l|l|l|}
\hline
 & $f$-vector of $\RBC(\M)$ &
$h$-vector of $\RBC(\M)$\\
\hline
pyramid $G$ & $(1,7,17,14)$ & $(1,4,6,3)$ \\
cube $H$ & $(1,11, 55, 159, 282, 290, 133)$ & $(1,5,15,29,40,32,11)$\\
\hline
\end{tabular}
\end{center}
\vspace{3mm}
\caption{\label{table:fh}The $f$-vectors and $h$-vectors of the broken circuit complexes of two graphs.}
\end{table}






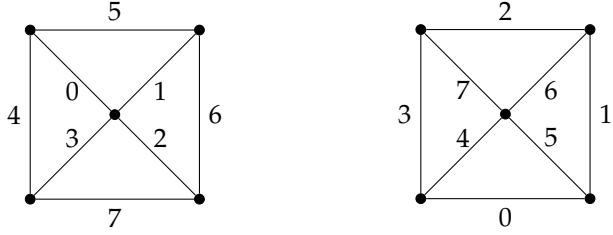
\begin{figure}[h]
\[
 \begin{tikzpicture}[baseline=(current bounding box.center),scale=0.75,
bend angle=25,vertex/.style={circle,draw,inner sep=1.3pt,fill=black}]

\node[vertex] (a) at (-0.5,-0.5) {};
\node[vertex] (b) at (2.5,-0.5) {};
\node[vertex] (c) at (2.5,2.5) {};
\node[vertex] (d) at (-0.5,2.5) {};
\node[vertex] (e) at (1,1) {};

\draw (a) to node[below] {7} (b);
\draw (b) to node[right] {6} (c);
\draw (c) to node[above] {5} (d);
\draw (d) to node[left] {4} (a);
\draw (a) to node[above] {3\ } (e);
\draw (b) to node[above] {\ 2} (e);
\draw (c) to node[below] {\ 1} (e);
\draw (d) to node[below] {0\ } (e);
\end{tikzpicture}
\qquad \qquad
\qquad
\begin{tikzpicture}[baseline=(current bounding box.center),scale=0.75,
bend angle=25,vertex/.style={circle,draw,inner sep=1.3pt,fill=black}]

\node[vertex] (a) at (-0.5,-0.5) {};
\node[vertex] (b) at (2.5,-0.5) {};
\node[vertex] (c) at (2.5,2.5) {};
\node[vertex] (d) at (-0.5,2.5) {};
\node[vertex] (e) at (1,1) {};

\draw (a) to node[below] {0} (b);
\draw (b) to node[right] {1} (c);
\draw (c) to node[above] {2} (d);
\draw (d) to node[left] {3} (a);
\draw (a) to node[above] {4\ } (e);
\draw (b) to node[above] {\ 5} (e);
\draw (c) to node[below] {\ 6} (e);
\draw (d) to node[below] {7\ } (e);
\end{tikzpicture}
\]
\caption{The graph $G$ of the pyramid and its dual graph $G^\perp$.}\label{fig:pyramid}
\end{figure}

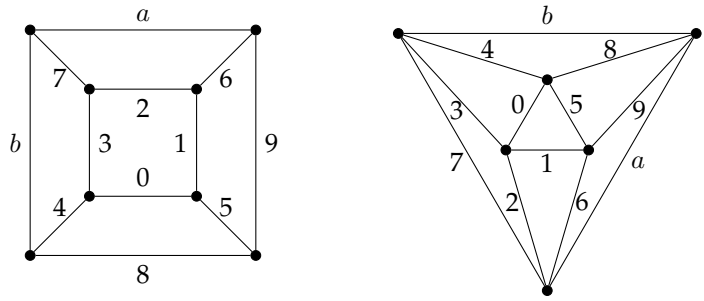
\begin{figure}[h]
\[
\begin{tikzpicture}[baseline=(current bounding box.center),scale=0.75,
bend angle=25,vertex/.style={circle,draw,inner sep=1.3pt,fill=black}]

\def\del{0.95}
\node[vertex] (a) at (-\del,-\del) {};
\node[vertex] (b) at (-\del,\del) {};
\node[vertex] (c) at (\del,\del) {};
\node[vertex] (d) at (\del,-\del) {};
\node[vertex] (e) at (-2,-2) {};
\node[vertex] (f) at (-2,2) {};
\node[vertex] (g) at (2,2) {};
\node[vertex] (h) at (2,-2) {};

\draw (a) to node[right] {3} (b);
\draw (b) to node[below] {2} (c);
\draw (c) to node[left] {1} (d);
\draw (d) to node[above] {0} (a);
\draw (e) to node[left] {$b$} (f);
\draw (f) to node[above] {$a$} (g);
\draw (g) to node[right] {9} (h);
\draw (h) to node[below] {8} (e);
\draw (a) to node[above] {4} (e);
\draw (d) to node[above] {5} (h);
\draw (c) to node[below] {6} (g);
\draw (b) to node[below] {7} (f);
\end{tikzpicture}
\qquad
\qquad
\begin{tikzpicture}[baseline=(current bounding box.center),scale=1.10,
bend angle=25,vertex/.style={circle,draw,inner sep=1.3pt,fill=black}]

\node[vertex] (a) at (-0.5,0) {};
\node[vertex] (b) at (0.5,0) {};
\node[vertex] (c) at (0,0.85) {};
\node[vertex] (d) at (1.8,1.41) {};
\node[vertex] (e) at (-1.8,1.41) {};
\node[vertex] (f) at (0,-1.7) {};

\draw (a) to node[below, yshift=.08cm] {1} (b);
\draw (b) to node[above right, xshift=-.1cm, yshift=-.1cm] {5} (c);
\draw (c) to node[above left, xshift=.1cm, yshift=-.1cm] {0} (a);
\draw (d) to node[above] {$b$} (e);
\draw (e) to node[left] {7} (f);
\draw (f) to node[right] {$a$} (d);
\draw (d) to node[below, xshift=-.04cm] {9} (b);
\draw (d) to node[below, yshift=.35cm, xshift=-0.15cm] {8} (c);
\draw (e) to node[above, xshift=.2cm, yshift=-.15cm] {4} (c);
\draw (e) to node[above, xshift=.06cm, yshift=-.5cm] {3} (a);
\draw (f) to node[above, xshift=-.19cm] {2} (a);
\draw (f) to node[above, xshift=.19cm
] {6} (b);

\end{tikzpicture}
\]
\caption{The graph $H$ of the cube and its dual graph $H^\perp$.
}\label{fig:cube}
\end{figure}

\subsection{Lagrangian combinatorics of matroids.}\label{sec:introgeom}

This paper explores the algebraic-combinatorial structure of the
conormal Chow ring $\AAA$ introduced in \cite[Section 3.5]{ADH}. 
The conormal Chow ring is an extension
of the Chow ring $A_{\M}$   studied in
\cite{FY04} and \cite{AHK}. 
We recall the central combinatorial notions from \cite{ADH}. 

\begin{enumerate}[$\bullet$]\itemsep 5pt
\item
A \emph{biflat} $F|G$ of $\M$ consists of a flat $F$ of $\M$ and a
flat $G$ of the dual matroid $\M^\perp$ such that they are nonempty, they are not both
equal to $E$, and their union is $E$. 

\item
Two biflats $F|G$ and $F'|G'$ of $\M$ are \emph{compatible} if
\[
\text{($F \subseteq F'$ and $G \supseteq G'$)} \ \  \text{or} \ \  \text{($F \supseteq F'$ and $G \subseteq G'$)}.
\]

\item
A \emph{biflag} of $\M$ is a collection $\F|\G$ of pairwise compatible biflats of $\M$ satisfying
\[
\bigcup_{F|G \in \F | \G} F \cap G \neq E.
\]

The \emph{length} of $\F|\G$ is $k$, the number of biflats it contains.  
\end{enumerate}

\begin{remark}
Let $\F$ be an increasing sequence of $k$ nonempty flats of $\M$, say
\[
\F=(\emptyset \subsetneq F_1 \subseteq \cdots \subseteq F_k \subseteq E),
\]
and let $\G$ be a decreasing sequence of $k$ nonempty flats of $\M^\perp$, say
\[
\G=( E \supseteq G_1 \supseteq \cdots \supseteq G_k \supsetneq \emptyset).
\]
Then the collection $\F|\G$ consisting of the pairs $F_1|G_1,\ldots,F_k|G_k$ is a biflag of $\M$ if and only if 
\[
\text{$F_j \cup G_j = E$ for all $1 \leq j \leq k$ and $F_j \cup
 G_{j+1} \neq E$  for some $0 \leq j \leq k.$}
\]
See \cite[Proposition 2.15]{ADH} for a straightforward verification. 
\end{remark}

\begin{example} \label{ex:biflatbiflag}
We write $\F|\G$ as a table with rows $\F$ and $\G$, augmented with the columns $\emptyset|E$ and $E|\emptyset$. 
For example, for the square pyramid graph $G$ of Figure \ref{fig:pyramid}, $01256|1347$
is a biflat and 
\[
\begin{array}{|c|c|c|}
\hline
\emptyset  & 01256 &  E \\
E &  1347  & \emptyset \\
\hline
\end{array}
 \subset  
\begin{array}{|c|cccc|c|}
\hline
\emptyset & 5 & 56 & 01256 & E & E \\
E & E &  E & 1347 & 3 & \emptyset \\
\hline
\end{array}
 \subset  
\begin{array}{|c|cccccc|c|}
\hline
\emptyset & 5 & 56 & 01256 & 01256 & E & E & E \\
E & E & E &  1347 & 347 & 347 & 3 & \emptyset \\
\hline
\end{array}
\]
are three biflags of biflats; the third one is maximal.
\end{example}

The biflats and biflags of a matroid encode the combinatorics of the
conormal fan, described in  \cite[Section 3.4]{ADH}. Our main
algebraic object of interest is the Chow ring of the conormal fan, which we
now define independently.  We begin with a polynomial ring $\SSS $ with  real coefficients and variables $x_{F|G}$
indexed by
the biflats of $\M$.  
For any set of biflats $\F|\G$, we consider the monomial
\[
\displaystyle x_{\F|\G} = \prod_{F|G \in \F|\G} x_{F|G}.
\]
We also define, for every element $i$ in the ground set  $E$, the linear forms%
\[
\gamma_i =\sum_{
i\in F, \, F\neq E} x_{F|G}, \qquad
\og_i =\sum_{
i\in G, \, G\neq E} x_{F|G}, \qquad
\delta_i =\sum_{
i\in F \cap G} x_{F|G}.
\]
These linear forms correspond to certain convex piecewise linear functions on the conormal fan of $\M$ \cite[Section 3.4]{ADH}.
We write
\begin{enumerate}[$\bullet$]\itemsep 5pt
\item $\III$ for the ideal generated by the monomials $x_{\F|\G}$, where $\F|\G$ is not a biflag, and
\item $\JJJ$ for the ideal generated by the linear forms $\gamma_i-\gamma_j$ and  $\og_i-\og_j$, for any $i$ and $j$ in $E$.
\end{enumerate}

\begin{definition}\label{def:Chowring} 
The \emph{conormal Chow ring} of  $\M$ is the quotient
\[
\AAA = \SSS/(\III + \JJJ).
\]
The equivalence classes of $\gamma_i$ and $\og_i$ in the conormal Chow ring do not depend on $i$.
We denote these classes by $\gamma$ and $\og$, respectively.
\end{definition}

We note that the equivalence class of  $\delta_i$ in the conormal Chow ring also do not depend on $i$:
For every biflat $F|G$, the element $i$ must be in $F$ or in $G$, and hence 
\[
\delta_i = \gamma_i + \og_i - \sum_{
F \neq E,  G\neq E} x_{F|G}.
\]
We write $\delta$ for the class of these elements in the conormal Chow ring. 
In \cite[Section 3.5]{ADH},
we constructed the \emph{degree map}
\[
\deg: A_{\M,\M^\perp}^{n-1} \longrightarrow \mathbb{R}, \qquad x_{\F|\G} \longmapsto x_{\F|\G}  \cap 1_{\M,\M^\perp}= \begin{cases}  1 &\text{if $\F|\G$ is a biflag,} \\ 0 &\text{if $\F|\G$ is not a biflag.} \end{cases}
\]
In \cite[Theorem 1.2]{ADH}, 
 we gave the following interpretation of $h(\BC(\M))$ in terms of the conormal intersection theory of $\M$:
\[
\deg(\gamma^k \delta^{n-k-1}) =\gamma^k \delta^{n-k-1} \cap 1_{\M,\M^\perp}= h_{r-k}(\BC(\M)) \ \ \text{for $0 \le k \le r$.}
\]
Our goal is to give a bijective proof of this numerical identity.

\textbf{Acknowledgments.}
The first author thanks the Mathematical Sciences Research Institute, the Simons Institute for the Theory of Computing, the Sorbonne Universit\'e, the Universit\`a di Bolo\-gna, and the Universidad de Los Andes for providing wonderful settings to work on this project, and Felipe Rinc\'on for valuable conversations on this topic; his research is partially supported by NSF grant DMS-1855610 and Simons Fellowship 613384.
The second author's research is supported by NSERC of Canada.
The third author is partially supported by NSF Grant DMS-2053308 and the Simons Investigator Grant.

\section{Biflats and biflags of matroids.}

Throughout the paper, we will fix a loopless and coloopless matroid $\M$ with ground set $E$ of size $n+1$ and rank $r+1$. The dual matroid $\M^\perp$ has rank $r^\perp+1 \coloneq  (n+1)-(r+1)=n-r$.

\subsection{Combinatorics of biflags.}\label{ss:conormal}

The following lemma is the combinatorial manifestation of the fact that the conormal fan is pure of dimension $n-1$. 
Its proof introduces several ideas that will be useful in what follows.

\begin{lemma}\label{lem:maximal}
Every maximal biflag of $\M$ has length $n-1$. 
\end{lemma}

\begin{proof}
Let  $\F | \G$ be a biflag consisting of biflats $F_1|G_1, \ldots, F_k|G_k$, say
\[
\F=(\emptyset \subsetneq F_1 \subseteq \cdots \subseteq F_k \subseteq E) \ \ \text{and} \ \ \G=(E \supseteq G_1 \supseteq \cdots \supseteq G_k \supseteq \emptyset).
\]
For every $i$, we write $r_i$ for the rank of $F_i$ in $\M$ and $r_i^\perp$ for the rank of $G_i$ in $\M^\perp$. 
We have
\[
\sum_{i=1}^{k+1} \left((r_i - r_{i-1}) + (r^\perp_{i-1} - r^\perp_i)\right) =
(r+1) + (r^\perp+1) = n+1.
\]
If $j$ is an index satisfying $F_j \cup G_{j+1} \neq E$, then we must have $F_j \neq F_{j+1}$ and $G_j \neq G_{j+1}$, so 
\[
(r_j - r_{j-1}) + (r^\perp_{j-1} - r^\perp_j) \geq 2.
\]
Since every summand of $n+1$ 
 is positive, there can be at most $n$ summands, so $k \leq n-1$.

We now show that $\F|\G$ is not maximal  if $k<n-1$. We consider two cases separately.


\begin{enumerate}[(1)]\itemsep 5pt
\item We have $r_i - r_{i-1} \geq 2$ or $r^\perp_{i-1} - r^\perp_i \geq 2$ for some $i$. 
\end{enumerate}

Suppose $r_i - r_{i-1} \geq 2$.
Choose any flat $F$ with $F_{i-1} \subsetneq F \subsetneq F_i$. 
If $F \cup G_i \neq E$, then $F|G_{i-1}$ is a biflat, and hence $\F^+ | \G^+ \coloneq  \F|\G \cup F|G_{i-1}$ is a biflag. 
On the other hand, if $F \cup G_i = E$, then $F|G_i$ is a biflat, and $\F^+ | \G^+ \coloneq \F|\G \cup F|G_i$ is a biflag. 


\begin{enumerate}[(1)]\itemsep 5pt
\item[(2)] We have $r_i - r_{i-1} \leq 1$ and $r^\perp_{i-1} - r^\perp_i \leq 1$ for all $i$. 
\end{enumerate}

In this case, $(r_i - r_{i-1}) + (r^\perp_{i-1} - r^\perp_i) \leq 2$
for all $i$, so at least two summands of $n+1$ 
are equal to 
$2$. Therefore, at least two values $i=j,j'$ satisfy $r_i - r_{i-1} =
r^\perp_{i-1} - r^\perp_i = 1$, and hence $F_{i-1} \subsetneq F_i$ and
$G_{i-1} \supsetneq G_i$. At least one of these, say $j'$, satisfies
$F_{j'-1} \cup G_j' \neq E$. Consider the other one; $F_j|G_{j-1}$ is
a biflat, and $\F^+ | \G^+ \coloneq  \F|\G \cup F_j|G_{j-1}$ is a biflag,
because it satisfies $F_{j'-1} \cup G_j' \neq E$.
\end{proof}



We will often use the following basic result \cite[Lemma 3.15]{ADH}. 

\begin{lemma}
\label{lem:big unions}
If $F$ and $G$ are nonempty flats of $\M$ and $\M^\perp$
respectively, then $\abs{F\cup G}\neq n$.
\end{lemma}

%
%
%
%

We close this section with some definitions that will be useful throughout the paper. Let $\F|\G$ be a biflag of $\M$. As above, we write $\F$ and $\G$ for the flags
\[
\F=(\emptyset \subsetneq F_1 \subseteq \cdots \subseteq F_k \subseteq E) \ \ \text{and} \ \ \G=(E \supseteq G_1 \supseteq \cdots \supseteq G_k \supseteq \emptyset),
\]
where $k$ is the length of $\F|\G$, and $F_0 = G_{k+1} = \emptyset$ and $F_{k+1} = G_0 = E$. 
 
 \begin{definition} \label{def:gaps}
The \emph{gap sequence} of $\F|\G$, denoted $\D(\F|\G)$,
is the sequence of \emph{gaps}
\[
 D_0|D_1|\cdots|D_k, \ \  \text{where }  \ \  D_j \coloneq  (F_{j+1} - F_j) \cap (G_j - G_{j+1}) 
 = E-(F_j\cup G_{j+1}).
 \]
 \end{definition}


The notion of gaps appears naturally in the model of the conormal fan as a
configuration space \cite[Section 2]{ADH}.  
 Each gap must have size at least $2$, as we noted above.

\begin{definition} \label{def:jumpsets}
The \emph{jump sets} of $\F|\G$ are
\begin{align*}
\j(\F)&=\set{j\mid 0\leq j\leq k \text{~and~} F_j \subsetneq F_{j+1}}, \\
\j(\G)&=\set{j\mid 0\leq j\leq k \text{~and~} G_j \supsetneq G_{j+1}}.
\end{align*}
The \emph{double jump set} is $\j(\F) \cap \j(\G)$; the elements of these sets are called \emph{jumps} of $\F$ and $\G$ and \emph{double jumps} of $\F|\G$, respectively.
\end{definition}

If the gap $D_j$ is nonempty, then $j$ must be a double jump. Thus, by definition, every biflag has at least one double jump.
The double jumps of a biflag $\F|\G$ play an important role throughout the paper, so we mark them 
$\blue{\subsetneq}$ and $\blue{\supsetneq}$ in the table for $\F|\G$. 

\begin{example} 
Consider the biflag $\F|\G$ given by the table 
\[
\begin{array}{|c|cccccc|c|}
\hline
\emptyset & 5 & 56 & \blue{\subsetneq} & 01256 & \blue{\subsetneq} & E & E \\
E & E &  E & \blue{\supsetneq} & 1347 &\blue{\supsetneq} & 3 & \emptyset \\
\hline
\end{array} \ .
\]
The biflag has jump sets $\j(\F) = \{0,1,2,3\}$ and $\j(\G) = \{2,3,4\}$.
The  set of  double jumps  is $\{2,3\}$, and the gap sequence is $\emptyset|\emptyset|02|47|\emptyset|$.
\end{example}

\begin{lemma}\label{lem:nongaps}
The gaps of $\F|\G$ are pairwise disjoint, and their union is
\[
\bigsqcup_{j=0}^k D_j = E - \bigcup_{i=1}^k \left(F_i \cap G_i\right).
\]
\end{lemma}

\begin{proof}
One easily verifies the first claim. For the second, first suppose $e\in F_i\cap G_i$.  Then $e\in F_j$ for all $j\geq i$, 
which means $e\not \in D_j$ for $i\leq j\leq k$.
Dually, $e\in G_j$ for all $j\leq i$, so $e\not\in D_j$ for all
$0\leq i\leq j-1$.
Now suppose $e$ is not in any gap. In this case, consider the index $1\leq i\leq k+1$ for which $e\in F_i-F_{i-1}$.  Since $e\in F_{i-1}\cup G_i$, we must have $e\in G_i$ and hence $e \in F_i\cap G_i$. 
\end{proof}

\subsection{Extended $\nbc$ biflags.}\label{sec:nbc cones}

In this section, we construct the  biflag $\F(B)|\G(B)$ and the extended biflag $\F^+(B)|\G^+(B)$ associated to each $\nbc$ basis $B$ of $\M$: These are the biflags we need to give a combinatorial formula for $\gamma^k \delta^{n-k-1}$ in Theorem \ref{thm:main}. 

We begin by recalling some basic facts about matroids on an ordered ground set.
For each basis $B\subseteq E$ of $\M$, we write
$
B^\perp \coloneq  E-B$
for the corresponding dual basis of $\M^\perp$. Each $i \notin B$ has a unique dependence on $B$, that is, there is a unique \emph{fundamental circuit} $C(B,i)$ contained in $B\cup i$ and containing $i$. 
Dually, for each $i\in B$, there is a unique \emph{fundamental cocircuit} in $B^\perp \cup{i}$
containing $i$, denoted $C^\perp(B,i)$.

\begin{definition}\label{def:activities}
The {\em externally active set} for $B$ in $\M$ is defined to be
\begin{align*}
\EA(B)&=\set{i \notin B \mid i=\min C(B,i)} \\
&=\set{i \notin B \mid  B \textrm{ is the lexicographically largest basis contained in } B \cup i}.
\end{align*}
Dually, the {\em internally active set} for $B$ in $\M$ is defined to be 
\begin{align*}
\IA(B) &= \set{i\in B \mid i=\min C^\perp(B,i)} \\
&= \set{i \in B \mid B \textrm{ is the lexicographically smallest basis containing } B - i}.
\end{align*}
The \emph{internally passive} and \emph{externally passive sets} of   $B$ are $\IP(B) = B - \IA(B)$ and $\EP(B) = B^\perp - \EA(B)$, respectively.
\end{definition}

Activites behave well with respect to matroid duality: 
The internally active set of $B$ in $\M$ equals the externally active set of $B^\perp$, and
the externally active set of $B$ in $\M$ equals the internally active set of $B^\perp$.
According to 
 \cite{Tutte67, Cr69},
the \emph{Tutte polynomial} of $\M$ equals
\[
T_{\M}(x,y) = \sum_{B}x^{|\IA(B)|}y^{|\EA(B)|}.
\]
In particular, for any natural numbers $i$ and $e$, the number $t_{i,e}$ of bases $B$ of $\M$ with $|\IA(B)|=i$ and $|\EA(B)|=e$ is independent of the ordering of the ground set $E$.

\begin{definition}\label{def:nbc}
We say a basis $B$ of $\M$ is a \emph{no broken circuit basis} or \emph{$\nbc$ basis} if $\EA(B)=\emptyset$.  This is equivalent to demanding that $B$ does not contain any \emph{broken circuit}, that is, any set of the form $C - \min C$ where $C$ is a circuit. 
\end{definition}

We now associate a biflag to each $\nbc$ basis of $\M$. We will later see how these biflags arise naturally in the Lagrangian combinatorics of ordered matroids.

\begin{proposition} \emph{($\nbc$ biflags)} \label{def:nbc biflags}
Let $B$ be a $\nbc$ basis with  $|\IA(B)| = k+1$. 
Define the sequence $\ee(B)=(e_1, \ldots, e_{r-k}, e_{r-k+1}, \ldots, e_{n-k-1})$ by the conditions
\[
B - \IA(B) = \{e_1 > \cdots > e_{r-k}\}  \ \ \text{and} \ \ 
 B^\perp - \min B^\perp = \{e_{r-k+1} < \cdots < e_{n-k-1}\}.
 \]
Let $\F(B)$ and $\G(B)$ be the flags of length $n-k-1$ given by 
\[
F_j | G_j = \begin{cases}
\cl\{e_1, e_2, \ldots, e_j\} | E & \text{for $1 \leq j \leq r-k$}, \\
E | \cl^\perp\{e_j, e_{j+1}, \ldots, e_{n-k-1}\} & \text{for $r-k+1 \leq j \leq n-k-1$,}
\end{cases}
\]
where $\cl$ and $\cl^\perp$ stand for the closure operators of $\M$ and $\M^\perp$.
Then $\F(B)| \G(B)$ is a biflag of $\M$.
\end{proposition}

The biflag
${\F(B)| \G(B)}$ is the \emph{$\nbc$ biflag} associated to $B$.  The table for ${\F(B)| \G(B)}$ reads 
\[
\hspace{-.5cm}
\begin{array}{|ccccccc|}
\hline
\cl(e_1)  & \cdots &  \cl(e_1, \ldots, e_{r-k}) &
\blue{\subsetneq} & E &  \cdots & E  \\
E  & \cdots &  E & 
\blue{\supsetneq} & \cl^\perp(e_{r-k+1}, \ldots, e_{n-k-1}) &  \cdots &  \cl^\perp(e_{n-k-1}) \\
\hline
\end{array}
\]
We call $x_{\F(B)|\G(B)}$ an \emph{$\nbc$ monomial} of $\M$.

\begin{proof}
Let $x$ be the minimum element of $B^\perp$.  We need to verify that 
\[
\cl(B-\IA(B)) \cup \cl^\perp(B^\perp-x) \neq E.
\]
Assume this is not the case. Since $x \notin \cl^\perp(B^\perp-x)$, we must have $x \in \cl(B-\IA(B))$. Thus the fundamental circuit $C(B,x)$ satisfies $C(B,x) \subseteq B-\IA(B) \cup x$. Since $B$ is $\nbc$, the minimum element of this circuit is some $y < x$. Also, since $y \in B-\IA(B)$, the minimum element in the fundamental cocircuit $C^\perp(B,y)$ is some $z < y$, and hence  $z<x$ as well. 
However, $C^\perp(B,y) \subseteq B^\perp \cup y$ shows that  $z \in B^\perp$, contradicting the minimality of $x$. 
\end{proof}

\begin{example}\label{ex:nbc biflag}
We illustrate the construction above with the example of Figure \ref{fig:cube}. This matroid has $n+1=12$ elements, rank $r+1=7$, and corank $r^\perp+1=5$. Let us order the ground set $0<1<\cdots <9<a<b$.
Consider the $\nbc$ basis 
$B=\green{01}5\green{6}78b$, 
whose internally active set $\IA(B) = \green{016}$ has $k+1=3$ elements, marked in green. We have $\ee(B) = (b,8,7,5;3,4,9,a)$, and ${\F(B)|\G(B)}$ is given by the table
\[
\begin{array}{|c|ccccccccc|c|}
\hline
\emptyset & {\bf b} & {\bf 8}b & {\bf 7}8b & {\bf 5}78b & \blue{\subsetneq} & E & E & E & E & E\\
E & E & E & E & E & \blue{\supsetneq} & 0{\bf 3}469a & {\bf 4}69a & 6{\bf 9}a & {\bf a} & \emptyset\\
\hline
\end{array} \ .
\]
For each proper flat $F_i$ and proper coflat  $G_i$, 
we have written in bold the new element $e_i$ that is not present in $F_{i-1}$ and $G_{i+1}$, respectively.
\end{example}

We now augment each $\nbc$ biflag to a maximal biflag containing it. 
Let $B$ be a $\nbc$ basis with
\[
\IA(B) = \{\green{c_1} > \cdots > \green{c_{k+1}}\}.
\]
Define the sequence $\ee(B)$ 
and the biflag $\F(B) | \G(B)$ as in Proposition \ref{def:nbc biflags},
and set
\[
S=B-\IA(B) \ \  \text{and} \ \ T = B^\perp - \min B^\perp. 
\]

\begin{proposition}\emph{(Extended $\nbc$ biflags)} \label{def:nbc regions}
Let $i$ be the largest index such that 
\[
c_i \notin \cl(S) \cup \cl^\perp(T),
\]
 or equivalently, the smallest index such that
\[
\cl(S, c_1, \ldots, c_i) \cup  \cl^\perp(T) = E.
\]
We define the flags $\F^+(B)$ and $\G^+(B)$ of length $n-1$ by inserting the following $k$ columns at the double jump between columns $r-k$ and $r-k+1$  of the table for ${\F(B)| \G(B)}$:
\begin{equation*}
\begin{array}{|ccccccc|}
\hline
\green{\cl(S, c_1)} &  \cdots &  \green{\cl(S, c_1, \ldots, c_{i-1})} &\blue{\subsetneq} &  \green{\cl(S, c_1, \ldots, c_{i-1})}  &  \cdots &   \green{\cl(S, c_1, \ldots, c_k)} \\
\green{E} &  \cdots &  \green{E} & \blue{\supsetneq} &  \green{\cl^\perp(T)} &  \cdots &  \green{\cl^\perp(T)} \\
\hline
\end{array}
\end{equation*}
Explicitly, we define $\F^+(B)|\G^+(B)$ by setting
\begin{align*}
F^+_j | G^+_j &= \begin{cases}
\cl\{e_1, e_2, \ldots, e_j\} | E & \textrm{ for } 1 \leq j \leq r-k, \\
\cl\{S, c_1, \ldots, c_{j-(r-k)}\} | E & \textrm{ for } 1 \leq j-(r-k) \leq i-1, \\
\cl\{S, c_1, \ldots, c_{j-(r-k)}\} | \cl^\perp(T) & \textrm{ for } i \leq j-(r-k) \leq k, \\
E | \cl^\perp\{e_{j-k}, e_{j-k+1}, \ldots, e_{n-k-1}\} & \textrm{ for } 1 \leq j-r \leq n-r-1.
\end{cases}
\end{align*}
Then $\F^+(B)|\G^+(B)$ is a maximal biflag of $\M$.
\end{proposition}

We call ${\F^+(B)|\G^+(B)}$ the \emph{extended $\nbc$ biflag} of the $\nbc$ basis $B$.

\begin{proof}
The only statement requiring proof is the equivalence between the two different definitions of $i$. Let $i$ be the largest index for which $c_i \notin \cl(S) \cup \cl^\perp(T)$.
We verify two statements:
\begin{enumerate}[(1)]\itemsep 5pt
  \item $\cl(S, c_1, \ldots, c_{i-1}) \cup  \cl^\perp(T) \neq E$. 
  \end{enumerate}
To see this, notice that $c_i \notin \cl(S, c_1, \ldots, c_{i-1})$ because $\{S, c_1, \ldots, c_{i-1}, c_i\} \subseteq B$ is independent. 
\begin{enumerate}[(1)]\itemsep 5pt
\item[(2)] $\cl(S, c_1, \ldots, c_i) \cup  \cl^\perp(T) = E$. 
\end{enumerate}
For this, let $\cl(S, c_1, \ldots, c_i) \cup  \cl^\perp(T) = U$, and consider any element $e \neq \min B^\perp$. 
\begin{enumerate}\itemsep 5pt 
  \item[(2-1)] If $e \in B^\perp - \min B^\perp = T$, then $e \in U$.
\item[(2-2)] If $e \in \IP(B) = S$, then $e \in S \subseteq \cl(S) \subseteq U$. 
\item[(2-3)] If $e \in \IA(B)$, then $e=c_j$ for some index $j$. If $j>i$, then by the maximality of $i$, we have  $c_j \in \cl(S) \cup \cl^\perp(T) \subseteq U$. If $j \leq i$, then $c_j \in \{S,c_1, \ldots, c_i) \subseteq U$. In either case, $e \in U$.
\end{enumerate}
We conclude that $U \supseteq E - \min B^\perp$. Since $\abs{U} \neq n$ by Lemma \ref{lem:big unions}, we must have $U=E$. 
\end{proof}

\begin{example} \label{ex:nbc region}
We continue with Example \ref{ex:nbc biflag}, which described the $\nbc$ biflag corresponding to the $\nbc$ basis $B=015678b$ for Figure \ref{fig:cube}. We now augment it to a maximal biflag. The set 
\[
\IA(B) = \green{016} = \{c_1 > c_2 > c_3\}
\]
 determines the two top row entries to be added, namely, 
 \[
 \cl(S, c_1) = 5678b \ \ \text{and} \ \  \cl(S, c_1, c_2) = 1256789ab.
 \]
  Notice that $i=2$ is the largest index for which $c_i \in E -  \cl(S) \cup \cl^\perp(T) = 12$
and also the smallest index such that $\cl(S, c_1, \ldots, c_i) \cup  \cl^\perp(T) = E$.
Therefore the bottom row entries switch from $E$ to $03469a$ between positions $i-1=1$ and  $i=2$:
\[
\begin{array}{|c|ccccccccccc|c|}
\hline
\emptyset & {\bf b} & {\bf 8}b & {\bf 7}8b & {\bf 5}78b & \green{5{\bf 6}78b} & \blue{\subsetneq} &  \green{{\bf 1}256789ab} & E & E & E & E & E\\
E & E & E & E & E & \green{E} & \blue{\supsetneq} & \green{03469a} & 0{\bf 3}469a & {\bf 4}69a & 6{\bf 9}a & {\bf a} & \emptyset \\
\hline
\end{array}.
\]
The resulting table corresponds to the extended $\nbc$ biflag ${\F^+(B)|\G^+(B)}$. 
\end{example}

\section{Lower bound for $\gamma^k\delta^{n-k-1}$.}\label{sec:canonical}


We aim to compute the degree of $\gamma^k\delta^{n-k-1}$ by expressing it as a sum of square-free monomials $x_{\F|\G}$.
One fundamental feature of the computation of $\gamma^k\delta^{n-k-1}$, which is simultaneously an advantage and a difficulty, is that there are many different ways to carry it out, since we have $n+1$ different definitions of $\gamma$ and $\delta$; namely $\gamma = \gamma_i$ and $\delta=\delta_i$ for every $i \in E$. It is not clear from the outset how one should organize this computation.

To have control over the computation, we require some structure  amidst that freedom. To achieve this, we introduce two key tools in this section:
\begin{enumerate}[$\bullet$]\itemsep 5pt
\item
(Definition \ref{def:canonical}) 
a canonical way of expanding powers of $\delta$, and
\item
(Lemma \ref{lem:aeradicates}) 
a criterion on a monomial $m$ that guarantees that $m \cdot \gamma^k = 0$.

\end{enumerate}
As we will see in Section \ref{sec:theproof}, the criterion in Lemma \ref{lem:aeradicates} shows that most expressions of the form $m \cdot \gamma^k$ vanish in degree $n-1$. The procedure in Definition \ref{def:canonical} will provide the combinatorial structure necessary to describe the terms that remain.

\subsection{The canonical expansion of $\delta^m$.} 

Let $\F|\G$ be a biflag and $x_{\F|\G}$ be the corresponding square-free monomial. Recall from Lemma
\ref{lem:nongaps} that the union of the gaps of $\F|\G$ is nonempty and equals
\[
D_0 \sqcup \cdots \sqcup D_k = E - \bigcup_{j=1}^k (F_j \cap G_j).
\]

\begin{definition}  \label{def:canonical0} (\emph{Canonical expansion of $x_{\F|\G} \, \delta$})
For a monomial $x_{\F|\G}$, let 
\[
e = e(\F|\G) \coloneq \max \big(E - \bigcup_{j=1}^k (F_j \cap G_j) \big),
\]
Define the \emph{canonical expansion} of $x_{\F|\G} \delta$ to be the expression
\[
x_{\F|\G} \,\delta_{e}
=
x_{\F|\G} \sum_{e \in F \cap G} x_{F|G}.
\]
\end{definition}

We recursively obtain the \emph{canonical expansion} of $x_{\F|\G} \delta^m$ by multiplying each monomial in the canonical expansion of  $x_{\F|\G} \delta^{m-1}$ by $\delta$, again using the canonical expansion. 

\begin{lemma}\label{lem:canonexpansion}
The canonical expansion of $x_{\F|\G} \, \delta$ is the sum of the monomials $x_{(\F \cup F) | (\G \cup G)}$ corresponding to the biflags of the form ${(\F \cup F) | (\G \cup G)} \supsetneq {\F|\G}$ such that $e \in F \cap G$. 
If $j$ is the unique index for which $e \in F_{j+1}-F_j$, then $e \in G_j - G_{j+1}$. Furthermore, the nonzero terms 
in the canonical expansion correspond to the biflats ${F|G}$ with $F_j \subseteq F \subseteq F_{j+1}$ and $G_j \supseteq G \supseteq G_{j+1}$. 
 \end{lemma}

\begin{proof}
The first statement follows directly from definition. For the second one, assume $e \in F_{j+1}-F_j$. Since $e \in F_{j+1}$ but $e \notin F_{j+1} \cap G_{j+1}$, we have $e \notin G_{j+1}$. Since $e \notin F_j$ but $F_j \cup G_j = E$, we have $e \in G_j$. Therefore $e \in G_j-G_{j+1}$ as desired.
Finally, if ${(\F \cup F) | (\G \cup G)}$ is a biflag with $e \in F \cap G$, then $e \notin F_j$ and $e \notin G_{j+1}$ imply that the biflat $F|G$ must be added in between indices $j$ and $j+1$ of ${\F|\G}$. Conversely, any such biflat arises in this expansion.
\end{proof}

\begin{definition}  \label{def:canonical} (\emph{Canonical expansion of $\delta^m$})
Given the canonical expansion $\delta^j = x_{\F_1|\G_1} + \cdots + x_{\F_t|\G_t}$, we compute the canonical expansion of $\delta^{j+1} = x_{\F_1|\G_1}\, \delta + \cdots + x_{\F_t|\G_t} \, \delta$ by adding the canonical expansions of the individual terms, following Definition \ref{def:canonical0}.
\end{definition}

We may think of the canonical expansion of $\delta^m$ as a recursive procedure to produce a list of biflags of length $m$, 
where each biflag is built up one biflat at a time according to the rules prescribed in Lemma \ref{lem:canonexpansion}.

\begin{example}\label{ex:pyramid2}
For the graph $G$ of the square pyramid in Figure \ref{fig:pyramid}, the canonical expansion of the highest nonzero power of $\delta$ in $A_{\M, \M^\perp}$, namely $\delta^{n-1} = \delta^6$, is
\begin{eqnarray*}
\delta^6 &=&  x_{6|E}\, x_{56|E}\, x_{4567|E}\, x_{E|23467}\, x_{E|347}\, x_{E|7} \\
&& + x_{7|E}\, x_{57|E}\, x_{4567|E}\, x_{E|23467}\, x_{E|36}\, x_{E|6} \\
&& + x_{7|E}\, x_{67|E}\, x_{4567|E}\, x_{E|235}\, x_{E|35}\, x_{E|5}.
\end{eqnarray*}
This expression is deceivingly short. Carrying out this seemingly simple computation by hand is very tedious; if one were to do it by brute force, one would find that 
the number of terms of the canonical expansions of $\delta^0, \ldots, \delta^6$ are the following:
\[
\begin{array}{|l||l|l|l|l|l|l|l|} \hline
& \delta^0 & \delta^1 & \delta^2 & \delta^3 & \delta^4 & \delta^5 & \delta^6\\ 
\hline \hline
\text{\# of monomials counted with multiplicity} & 1 & 29 & 352 & 658 & 383 & 69 & 3\\ \hline
\text{\# of distinct monomials} & 1 & 29 & 333 & 621 & 370 & 68 & 3\\ \hline
\end{array} \ .
\]
This example shows typical behavior: for small $k$ the number of biflags in the 
expansion of $\delta^k$ increases with $k$, but as $k$ approaches
$n-1$, increasingly many products $x_{\F|\G} \, \delta$ are zero, and the canonical expansions become shorter.
\end{example}

Each monomial contribution $x_{\F|\G}$ to the canonical expansion of $\delta^m$ is built up through a sequence of monomials 
\[
1=x_{\F^0|\G^0},  \ \  x_{\F^1|\G^1}, \ \  \ldots, \ \   x_{\F^m|\G^m} = x_{\F|\G},
\]
 where $x_{\F^{k+1}|\G^{k+1}}$ appears in the canonical expansion of $x_{\F^k|\G^k} \delta = x_{\F^k|\G^k} \delta_{e(\F^k|\G^k)}$. We relabel the $e(\F^k|\G^k)$s, writing
$e_i = e({\F^k|\G^k})$ if the $i$-th biflat of the final biflag $\F|\G$ is obtained in the canonical expansion of $x_{\F^k|\G^k} \delta$. Thus both the flat and the coflat of that $i$-th biflat must contain $e_i$.
We call the resulting sequence $\ee = (e_1, \ldots, e_m)$ the \emph{arrival sequence} of this contribution.
We summarize the definition of the canonical expansion of $\delta^m$ and the arrival sequence of each term in the following proposition.

\begin{proposition}\label{prop:expansion}
The canonical expansion of $\delta^m$ 
 consists of the monomials $x_{\F|\G}$ indexed by the collection $\T^m_{\M,\M^\perp}$ of all pairs $(\F|\G , \ee)$ for which
\begin{enumerate}[(1)]\itemsep 5pt
\item 
the biflag $\F|\G = \{F_1|G_1, \ldots, F_m|G_m\}$ is a biflag of biflats of $\M$, and
\item 
the arrival sequence $\ee=(e_1, \ldots, e_m)$ is a sequence of distinct elements of $E$ such that 
\[
e_i\in F_i\cap G_i \ \ \text{and} \ \  e_i=\max\big(E-\bigcup_{j\colon e_j > e_i} (F_j\cap G_j)\big) \ \ \text{for all $1 \le i \le m$.}
\]
\end{enumerate}
 In symbols, the following identity holds in the conormal Chow ring of $\M$:
\[
\delta^{m}=\sum_{(\F|\G , \ee)
\in{\T^m_{\M,\M^\perp}}} x_{\F|\G}. 
\]
\end{proposition}


We record each pair $(\F|\G , \ee) \in \T^m_{\M,\M^\perp}$ as a table, where we expand the table of biflag $\F|\G$ by placing element $e_i$ directly below $F_i$ and $G_i$.
Each contribution to the canonical expansion $\delta^m$ comes from such a table.
\[
\begin{array}{|cc|ccccccccccc|cc|}
\hline
\emptyset & \subsetneq & F_1 & \subseteq & \cdots & \subseteq & F_d &
\subseteq & F_{d+1} & \subseteq & \cdots & \subseteq  & F_m & \subseteq & E\\
E & \supseteq & G_1 & \supseteq & \cdots & \supseteq & G_d & 
\supseteq & G_{d+1} & \supseteq & \cdots & \supseteq  & G_m & \supsetneq & \emptyset\\
\hline
&&e_1 && \cdots && e_d && e_{d+1}&& \cdots && e_m&&\\
\hline
\end{array} \, 
\]
The canonical expansion of $\delta^m$ may contain repeated terms $x_{\F|\G}$ which lead to different tables $(\F|\G;\ee)$ with the same  biflat $\F|\G$ but different arrival sequences $\ee$.



\begin{example}
Let us revisit the canonical expansion of $\delta^6$, the highest power of $\delta$, in Example~\ref{ex:pyramid2}. The first monomial arises from the following table.
\[
 \begin{array}{|cc|ccccccccccc|cc|}
\hline
\emptyset & \subset&  6 & \subsetneq & 56 & \subsetneq &  4567 & \blue{\subsetneq} & E & = & E & = & E & = &  E \\
{}E & = & E & = & E & = & E & \blue{\supsetneq} & 23467 & 
\supsetneq & 347 & \supsetneq & 7 & \supset&  \emptyset \\ \hline
&&e_1=6 & & e_2= 5 && e_3= 4 && e_4= 2 && e_5= 3 && e_6= 7 &&\\
\hline
\end{array}
\]
The terms $x_{F_i|G_i}$ arrive to the monomial $x_{\F|\G}$ in the descending order of the $e_i$s, namely, 
\[
x_{E|7}x_{6|E}x_{56|E}x_{4567|E}x_{E|347}x_{E|23467}.
\]
 This means that 
$x_{E|7}$ is in the canonical expansion of $\delta = \delta_7$, 
$x_{E|7}x_{6|E}$ is in the canonical expansion of $x_{E|7}\delta = x_{E|7}\delta_6$, 
$x_{E|7}x_{6|E}x_{56|E}$ is in the canonical expansion of $x_{E|7}x_{6|E}\delta=x_{E|7}x_{6|E}\delta_5$, and so on.
The corresponding biflag is $\F(B)|\G(B)$ for the $\nbc$ basis $B=\green{0}456$ with $\IA(B) = \green{0}$, as introduced in Definition \ref{def:nbc biflags}. 

The two other monomials in our expansion of $\delta^6$ are $x_{7|E}\, x_{57|E}\, x_{4567|E}\, x_{E|23467}\, x_{E|36}\, x_{E|6}$ and $x_{7|E}\, x_{67|E}\, x_{4567|E}\, x_{E|235}\, x_{E|25}\, x_{E|5}$, which correspond to the biflags of $\green{0}457$ and $\green{0}467$, respectively. These are the other two $\nbc$ bases whose only internally active element is $0$. 
Note that we must have $0 \in \IA(B)$ for any $\nbc$ basis $B$.

This example illustrates a general phenomenon: The case $k=0$ of  Theorem \ref{thm:main} says that the terms of the canonical expansion of $\delta^{n-1}$ in the conormal Chow ring correspond to the $\nbc$ biflags of the $\nbc$ bases $B$ of $\M$ with $|\IA(B)|=1$; these are also known as the \emph{$\beta$-$\nbc$ bases} of $\M$. They are enumerated by Crapo's beta invariant $\beta_{\M} = h_r(\BC(\M)) = t_{1,0}(\M)$: see \cite{Zie92}.
\end{example}

\subsection{$h_{r-k}$-many $\nbc$ monomials in $\delta^{n-k-1}$.}


To each $\nbc$ basis $B$ of $\M$, we associated a sequence $\ee(B)$ and a biflag ${\F(B)| \G(B)}$ in Definition \ref{def:nbc biflags}. We will now show that the table $(\F(B)| \G(B), \ee(B))$ satisfies the conditions of Proposition \ref{prop:expansion}, and hence the $\nbc$ monomial $x_{\F(B)|\G(B)}$ appears in the canonical expansion of $\delta^{n-k-1}$ with arrival sequence $\ee(B)$, where $k+1 = |\IA(B)|$.

\begin{example} \label{ex:partialnbcmonomial}
For the cube graph of Figure \ref{fig:cube} and $k=2$, the $\nbc$ basis $\green{01}5\green{6}78b$ with $\IA(B) = \green{016}$ gives rise to the $\nbc$ biflag of Example \ref{ex:nbc biflag}, and the table
\[
\begin{array}{|c|ccccccccc|c|}
\hline
\emptyset & b & 8b & 78b & 578b & \blue{\subsetneq} & E & E & E & E & E\\
E & E & E & E & E & \blue{\supsetneq} &  03469a & 469a & 69a & 9 & \emptyset \\
\hline
& b & 8 & 7 & 5  & & 3 & 4 & 9 & a & \\
\hline
\end{array}
\]
gives rise to the following $\nbc$ monomial in the canonical expansion of $\delta^8$:
\[
x_{b|E} \, x_{8b|E} \, x_{78b|E} \, x_{578b|E} \, x_{E|03469a} \, x_{E|469a} \, x_{E|69a} \, x_{E|9}.
\]
\end{example}

\begin{proposition}\label{prop:delta expansion III}
If $B$ is a $\nbc$ basis of $\M$ with $|I(B)| = k+1$, then the $\nbc$ monomial $x_{\F(B)| \G(B)}$ arises in the canonical expansion of $\delta^{n-k-1}$. 
\end{proposition}

\begin{proof}
We verify that $({\F(B)| \G(B)},\ee(B))$ satisfies the conditions of Proposition \ref{prop:expansion}. 
We know that ${\F(B)| \G(B)}$ is a biflag by Proposition \ref{def:nbc biflags}, so it remains to show that 
\[
e=\max \big(E- \cl(\IP(B)_{>e}) - \cl^\perp(B^\perp_{>e})\big) \ \ \text{for any $e \in \IP(B)$ or $e \in B^\perp-\min B^\perp$.}
\]
First, we show that $e \in E- \cl (\IP(B)_{>e}) - \cl^\perp(B^\perp_{>e})$, considering two cases:
\begin{enumerate}[(1)] \itemsep 5pt
\item $e \in B^\perp-\min B^\perp$: In this case, $e \notin \cl^\perp(B^\perp_{>e})$ since $B^\perp$ is independent in $M^\perp$. Also, if we had $e \in \cl (\IP(B)_{>e})$, then the fundamental circuit $C(B,e)$ would be contained in $\IP(B)_{>e} \cup e$, and hence its smallest element would be $e$. Thus $e$ would be externally active in $B$, contradicting the assumption that $B$ is a $\nbc$ basis. 
\item $e \in \IP(B)$: In this case, $e \notin  \cl (\IP(B)_{>e})$ since $\IP(B) \subseteq B$ is independent in $\M$. Also, if we had $e \in \cl^\perp(B^\perp_{>e})$, then the fundamental cocircuit $C^\perp(B,e)$ would be contained in $B^\perp_{>e} \cup e$, and hence its smallest element would be $e$, contradicting that $a$ is internally passive in $B$.
\end{enumerate} 
Now, for the sake of contradiction, let us assume that $\max \big(E- \cl(\IP(B)_{>e}) - \cl^\perp(B^\perp_{>e})\big) = f > e$. We consider three cases:
\begin{enumerate}[(1)]\itemsep 5pt 
\item
$f \in B^\perp$: Then $f>e$ would imply  $f \in B^\perp_{>e} \subseteq \cl^\perp(B^\perp_{>e})$. 
\item
$f \in \IP(B)$: Then $f>e$ would imply  $f \in \IP(B)_{>e} \subseteq \cl(\IP(B)_{>e})$.
\item
$f \in \IA(B)$: This means that $f = \min C^\perp(B,f)$, so this fundamental cocircuit is contained in $B^\perp_{>f} \cup f$. Thus $f \in \cl^\perp (B^\perp_{>f})$ while
  $f \notin \cl^\perp(B^\perp_{>e})$, which is impossible since $B^\perp_{>f} \subseteq B^\perp_{>e}$.
\end{enumerate}
This completes the proof.
\end{proof}

%
%
%

\subsection{Multiplying by a high power of $\gamma$ eradicates non-initial monomials.}

\begin{lemma}\label{lem:multbya}
Consider a monomial $x_{\F|\G}$ and let $l$ be the largest index for which $F_l \neq E$. 
If $c \notin F_l$, then $x_{\F|\G} \, \gamma_c$ is the sum of the monomials $x_{(\F \cup F) | (\G \cup G)}$ corresponding to 
the variables $x_{F|G}$ with $F_l \cup c \subseteq F \subsetneq E$ and $G_l \supseteq G \supseteq G_{l+1}$.
 \end{lemma}

\begin{proof}
We have 
$
x_{\F|\G} \,\gamma_{c}
=
x_{\F|\G} \sum_{
c \in F \neq E} x_{F|G}$.
Now, if ${(\F \cup F) | (\G \cup G)}$ is a biflag with $c \in F \neq E$, then we must have $F_l \subsetneq F \subsetneq E = F_{l+1}$, so the biflat $F|G$ must be added in between indices $l$ and $l+1$ of ${\F|\G}$. Conversely, any such biflat arises in this expansion.
\end{proof}

\begin{definition}
Let us call a monomial $x_{\F|\G}$ \emph{initial} if 
the distinct flats in $\F$ have ranks $1, 2, \ldots, i,$ and $r+1$ for some $i$. 
\end{definition}

The following technical lemma will play an important role. 
It shows that the multiplication by a high power of $\gamma$ eradicates non-initial monomials.

\begin{lemma} \label{lem:aeradicates} 
Let $\F|\G$ be a biflag of $\M$, and let $s$ be the number of distinct proper flats in $\F$.
\begin{enumerate}[(1)]\itemsep 5pt
\item\label{lem:aeradicates:1}
If $s+k > r$, then 
$x_{\F|\G} \, \gamma^k=0$.
\item\label{lem:aeradicates:2}
If $s+k=r$ and $\F$ is not initial, then  $x_{\F|\G} \, \gamma^k = 0$.
\end{enumerate}
\end{lemma}

\begin{proof}
The proof of the first part is nearly identical and simpler than the proof of the second part.
For the second part, as in Lemma \ref{lem:multbya}, let $l$ be the largest index for which $F_l \neq E$, and let $c \notin F_l$.
We proceed by descending induction on $s$. 

The largest possible value of $s$ for a non-initial flag is $r-1$.
Suppose first that $s=r-1$ and $k=1$. If $\F$ is not initial, we must have $\rank F_{l} = r$. Lemma \ref{lem:multbya} then implies that 
\[
x_{\F|\G} \,\gamma = x_{\F|\G} \,\gamma_{c} = \sum_{\substack{
F_{l} \cup c \subseteq F \neq E \\ G_l \supseteq G \supseteq G_{l+1}}} x_{(\F \cup F) | (\G \cup G)} = 0,
\]
since the only flat containing $F_{l}$, which has corank 1, and $c$, which is not in $F_{l}$, is $E$. 

Now suppose that the result is true for some value of $s \leq r-1$, and consider any biflag $\F|\G$, where $\F$ has  $s-1$ distinct proper flats. We have
\[
x_{\F|\G}\, \gamma^{r-(s-1)} = 
(x_{\F|\G}\, \gamma_c)\gamma^{r-s} = 
\sum_{\substack{
F_{l} \cup c \subseteq F \neq E \\ G_l \supseteq G \supseteq G_{l+1}}} x_{(\F \cup F) | (\G \cup G)} \gamma^{r-s} = 0,
\]
where each summand is $0$ by the inductive hypothesis because each flag $\F \cup F$ that arises is not initial and has $s$ distinct proper flats.
\end{proof}

\subsection{$h_{r-k}$-many  extended $\nbc$ monomials in $\gamma^k\delta^{n-k-1}$.} \label{sec:nbcmonomials}

%

Lemma \ref{lem:aeradicates}
 shows that many monomials in $\delta^{n-k-1}$ are eliminated when one multiplies them by $\gamma^k$. Let us now show that each one of the $h_{r-k}(\BC(\M))$ $\nbc$ monomials $x_{\F(B)|\G(B)}$ of Proposition \ref{prop:delta expansion III} resists  multiplication by $\gamma^k$, and gives rise to its corresponding extended  $\nbc$ monomial $x_{\F^+(B)|\G^+(B)}$ in  $\gamma^k \delta^{n-k-1}$, as introduced in Definition \ref{def:nbc regions}. We will later see that these are the only monomials that resist  multiplication by $\gamma^k$.
 
 \begin{proposition} \label{prop:nbc is resistant}
For every $\nbc$ basis $B$ with $\abs{\IA(B)} = k+1$, we have 
\[
x_{\F(B)|\G(B)} \, \gamma^k = x_{\F^+(B)|\G^+(B)}.
\]
\end{proposition}

\begin{proof}
Recall from Proposition \ref{def:nbc biflags} that, if we write
\[
S = B - \IA(B) = \{e_1 > \cdots > e_{r-k}\} \ \ \text{and} \ \ 
T = (E-B) - \min(E-B) = \{e_{r-k+1} < \cdots < e_{n-k-1}\},
\]
then the table of the $\nbc$ monomial $x_{\F(B)|\G(B)}$ is
\[
\begin{array}{|cc|ccccccccccc|cc|}
\hline
\emptyset & \subsetneq &  F_1  & \subsetneq & \cdots & \subsetneq & F_{r-k} & \blue{\subsetneq} &  E &  = & \cdots & =  & E & = & E\\
E & = & E & = & \cdots & = & E & \blue{\supsetneq} & G_{r-k+1} & \supsetneq & \cdots & \supsetneq  & G_{n-k-1} & \supsetneq & \emptyset
\\
\hline
&& e_1  & > & \cdots &> & e_{r-k} &  & e_{r-k+1}&<& \cdots &<& e_{n-k-1} &&
\\
\hline
\end{array}
\ ,
\]
where  $F_j = \cl(e_1, \ldots, e_j)$ for $j \leq r-k$ and $G_j = \cl^\perp(e_j, \ldots, e_{n-k-1})$ for $j>r-k$. In particular, 
\[
F_{r-k} = \cl(S) \ \  \text{and}  \ \ G_{r-k+1}=\cl^\perp(T).
\]
We write
$\IA(B) = \{\green{c_1}>\green{c_2}>\cdots>\green{c_{k+1}}\}$, and
 multiply $x_{\F|\G}$ by $\gamma_{c_1}, \ldots, \gamma_{c_k}$ to compute
\[
x_{\F|\G}\gamma^k = x_{\F|\G}\gamma_{c_1}\gamma_{c_2}\cdots \gamma_{c_k}.
\]

Since $c_1 \notin F_{r-k}$, when we multiply $x_{\F|\G} \, \gamma_{c_1}$, Lemma \ref{lem:multbya} tells us that every resulting term $x_{(\F \cup F) | (\G \cup G)}$ corresponds to a biflat $F|G$ such that $F_{r-k} \cup c_1 \subsetneq F \subsetneq E$. Also, for such a term to resist further multiplication by $\gamma^{k-1}$, $\F \cup F$ must be an initial flag by  Lemma
\ref{lem:aeradicates}, so $\rank(F) = r-k+1$. This implies that $F = \cl(S, c_1)$; we denote $F$ by $F_{r-k+1}$. 

Similarly, since $c_2 \notin F_{r-k+1}$, every term in $x_{\F|\G} \, \gamma_{c_1}\gamma_{c_2}$ that resists further multiplication by $\gamma^{k-2}$ must introduce $F_{r-k+2}\coloneq  \cl(S, c_1, c_2)$ as the new flat in the flag. Continuing with this line of reasoning, we see that the nonzero terms in $x_{\F|\G} \, \gamma_{c_1}\gamma_{c_2}\cdots \gamma_{c_k}$ are obtained from the table of $x_{\F(B)|\G(B)}$ 
by adding new columns between columns $r-k$ and $r-k+1$ as 
\[
\begin{array}{|ccccccccccc|}
\hline
\cdots & \cl(S) & \subsetneq & \green{\cl(S, c_1)} & \subsetneq & \cdots & \subsetneq & \green{\cl(S, c_1, \ldots, c_k)} & \subsetneq & E & \cdots \\
\cdots &  E & \supseteq & \green{G'_{r-k+1}} & \supseteq & \cdots & \supseteq & \green{G'_r} & \supseteq & \cl^\perp(T) & \cdots \\
\hline
& e_{r-k} & & \green{c_1} & > & \cdots & > & \green{c_k} & & e_{r-k+1} & 
\\
\hline
\end{array}
\ ,
\]
for any choice of coflats $G'_{r-k+1} \supseteq \cdots \supseteq G'_r$, each of which is either $E$ or $\cl^\perp(T)$, since $r^\perp(T) = r^\perp - 1$.

The only freedom we appear to have left is the choice of the unique value of $1 \leq i \leq k+1$ for which $G'_{r-k+i-1} = E$ and $G'_{r-k+i} =  \cl^\perp(T)$:

\begin{equation*}
\hspace{-.5cm}
\begin{array}{|ccccc|}
\hline
\cdots & \green{\cl(S, c_1, \ldots, c_{i-1})} &\blue{\subsetneq} &  \green{\cl(S, c_1, \ldots, c_i)}  & \cdots  \\
\cdots & \green{E} & \blue{\supsetneq} &  \green{\cl^\perp(T)} & \cdots  \\
\hline
\end{array}
\end{equation*}

\noindent However, for this to be a valid flag, we must have 
\[
\cl(S, c_1, \ldots, c_{i-1}) \cup \cl^\perp(T) \neq E \ \  \text{and} \ \  \cl(S, c_1, \ldots, c_i) \cup \cl^\perp(T) = E. 
\]
Therefore, we do not have the freedom to choose $i$; it must be the smallest index for which $\cl(S, c_1, \ldots, c_i) \cup \cl^\perp(T) = E$. 

This precisely matches the description of the extended $\nbc$ biflag of $B$ in Definition \ref{def:nbc regions}. Thus the unique surviving term in $x_{\F(B)|\G(B)}\gamma_{c_1}\gamma_{c_2}\cdots \gamma_{c_k}$ is precisely $x_{\F^+(B)|\G^+(B)}$, as we wished to show.
\end{proof}

\begin{example} \label{ex:nbcmonomial}
Let us return to Example \ref{ex:partialnbcmonomial} and multiply
 the $\nbc$ monomial of $B=\green{01}5\green{6}78b$ with $\IA(B) = \green{016}$ by $\gamma^2$ by computing 
 \[
(x_{b|E} \, x_{8b|E} \, x_{78b|E} \, x_{578b|E} \,  x_{E|03469a} \, x_{E|469a} \, x_{E|69a} \, x_{E|9}) \cdot  \gamma_6 \cdot \gamma_1.
 \]
 This adds two new columns to the middle of the table in  Example \ref{ex:partialnbcmonomial}, shown in green. The new entries in the top row must be $\cl(578b\green{6}) = 5678b$ and $\cl(578b\green{61}) = 1256789ab$. The new entries in the bottom row can equal $E$ or $03469a$. Since $5678b \cup 03469a \neq E$, and we must have $F_i \cup G_{i} \neq E$ for all $i$, the first one must be $E$. 
Since $1256789ab \cup 03469a = E$, and we must have $F_i \cup G_{i+1} \neq E$ for some $i$, the second one must be $03469a$. Thus the only possible table is
 \[
\begin{array}{|c|cccccccccc|c|}
\hline
\emptyset & b & 8b & 78b & 578b & \green{5678b} & \green{1256789ab} & E & E & E & E & E\\
E & E & E & E & E & \green{E} & \green{03469a} & 03469a & 469a & 69a & 9 & \emptyset \\
\hline
& b & 8 & 7 & 5 & \green{6} & \green{1} & 3 & 4 & 9 & a & \\
\hline
\end{array} \ ,
\]
and the resulting monomial is the extended  $\nbc$ monomial of $B$
\[
x_{b|E} \, x_{8b|E} \, x_{78b|E} \, x_{578b|E} \, \green{x_{5678b, E} \, x_{1256789ab, 03469a}} \, 
 x_{E|03469a} \, x_{E|469a} \, x_{E|69a} \, x_{E|9},
\]
 corresponding to the extended $\nbc$ biflag of Example \ref{ex:nbc region}.
\end{example}

We have now made some progress towards our proof of Theorem \ref{thm:main}, which says that 
\[
\gamma^k \delta^{n-k-1} = \sum_{\substack{B \nbc \textrm{basis} \\ |\IA(B)| = k+1}} x_{\F^+(B)| \G^+(B)}.
\]
Propositions \ref{prop:delta expansion III} and \ref{prop:nbc is resistant} show that all terms in the right hand side of this expression do arise in $\gamma^k \delta^{n-k-1}$. Proving that no other terms appear requires significantly more work; this is the content of Section \ref{sec:theproof}.

\section{Upper bound for $\gamma^k\delta^{n-k-1}$.}\label{sec:theproof}


Proposition \ref{prop:expansion} shows that the terms $x_{\F|\G}$ in the canonical expansion of $\delta^{n-k-1}$ are given by the set $\T_{\M,\M^\perp}$ of combinatorially determined tables $(\F|\G,\ee)$. Lemma \ref{lem:aeradicates} shows that multiplication by high powers of $\gamma$ eradicates many of these monomials. The main goal of this section will be to characterize those monomials in the canonical expansion of $\delta^{n-k-1}$ that resist multiplication by $\gamma^k$.

\begin{definition}
A monomial $x_{\F|\G}$ of degree $n-k-1$ and the corresponding to table $(\F|\G,\ee) \in \T_{\M,\M^\perp}$ 
are said to be  \emph{resistant} if 
$x_{\F|\G}\, \gamma^k \neq 0$
in the conormal Chow ring of $\M$.
\end{definition}

We saw in Proposition \ref{prop:nbc is resistant} that any $\nbc$ monomial does resist multiplication by $\gamma^k$, and gives rise to its corresponding extended $\nbc$ monomial. We will eventually show in Proposition \ref{prop:delta expansion II} that these are in fact the only resistant monomials, and hence that $\deg(\gamma^k\delta^{n-k-1}) = h_{r-k}(\BC(\M))$. 
This proof will require several steps, which we carry out in the following subsections.

\subsection{The jump sets of a resistant term of $\delta^{n-k-1}$.}

Recall the notion of jump sets of  $\F|\G$ from Definition \ref{def:jumpsets}. 
We write $\Tr_k \M$ for the $k$-th truncation of $\M$.

\begin{lemma} \label{lem:full flags}
If the monomial $x_{\F|\G}$ arises in the canonical expansion of $\delta^{n-k-1}$ and resists multiplication by $\gamma^k$,  then
\begin{enumerate}[(1)]\itemsep 5pt
\item
$\F|\G$ has a unique double jump, and
\item 
$\F$ and $\G$ are complete flags of nonempty flats in $\Tr_k \M$ and $\M^\perp$, with possible repetitions.
\end{enumerate}
\end{lemma}
%
%
%

\begin{proof}
Since $x_{\F|\G} \gamma^k$ is nonzero,  the first part of Lemma \ref{lem:aeradicates} tells us that $\F$ contains $s \leq r-k$ distinct proper flats, so $\abs{\j(\F)} =s+1 \leq r-k+1$. Also, since $\M^\perp$ has rank $r^\perp+1 = n-r$,
we have $\abs{\j(\G)}\leq n-r$.
 On the other hand, since a square-free monomial cannot contain repeated terms, $\j(\F)\cup \j(\G)= \set{0, 1,\ldots,n-k-1}$. 
Therefore, 
\[
n-k \leq 
 \abs{\j(\F) \cup \j(G)}+\abs{\j(\F) \cap \j(\G)} = \abs{\j(\F)}+\abs{\j(\G)}  \leq n-k+1,
\]
and the number of double jumps is $\abs{\j(\F)\cap \j(\G)}\leq1$. But $\F|\G$ has at least one nonempty gap $D_d$, which guarantees a double jump $d$.


The above analysis also implies that $\F$ contains $s=r-k$ distinct proper flats --  which must have ranks $1, 2, \ldots, r-k$ by the second part of Lemma \ref{lem:aeradicates} -- 
and that $\G$ has flats in every rank of $\M^\perp$. 
\end{proof}

It follows that any table $(\F|\G,\ee)$ arising in the canonical expansion of $\delta^{n-k-1}$ that resists multiplication by $\gamma^k$ has the form
\[
\begin{array}{|cc|ccccccccccc|cc|}
\hline
\emptyset & \subsetneq & F_1 & \subseteq & \cdots & \subseteq & F_d &
\blue{\subsetneq} & F_{d+1} & \subseteq & \cdots & \subseteq  & F_{n-k-1} & \subseteq & E\\
E & \supseteq & G_1 & \supseteq & \cdots & \supseteq & G_d & 
\blue{\supsetneq} & G_{d+1} & \supseteq & \cdots & \supseteq  & G_{n-k-1} & \supsetneq & \emptyset\\
\hline
&&e_1 && \cdots && e_d & \blue{D} & e_{d+1}&& \cdots && e_{n-k-1}&&\\
\hline
\end{array} \,  
\]
where $d$ is the unique double jump of $\F|\G$. 
We write
\[
D\coloneq  D_d =  E-(F_d \cup G_{d+1}) = (F_{d+1}-F_d) \cap (G_d - G_{d+1})
\]
for the unique nonempty gap. In every column other than the $d$-th, one inclusion is strict and the other one is an equality. 
From now on, we will also record the nonempty gap $D$ in the bottom row of the table $(\F|\G;\ee)$. This is redundant information, but it will be useful visually in the proofs that follow.

%

\begin{remark}
For each index $i$, the flat $F_i$ contains the bottom row entries below it and to its left; namely, $e_1, \ldots, e_i$, and $D$ if $i \geq d+1$. Similarly, the coflat  $G_i$ contains the bottom row entries below it and to its right; $e_i, \ldots, e_{n-k-1}$, and $D$ if $i \leq d$.
\end{remark}

We continue with two more easy but important properties of the canonical expansion. The first one tells us that the arrival sequence $\ee$ and the index $d$ of the double jump completely determine which inclusions are strict in the table of $(\F|\G,\ee)$. We define the \emph{descent set} and \emph{ascent set} of $\ee$ by
\[
\Des(\ee) = \{i \, : \, e_i > e_{i+1}\} \ \ \text{and} \ \ 
\Asc(\ee) = \{i \, : \, e_i < e_{i+1}\}.
\]

\begin{lemma}\label{lem:ordering}
If $i\in \j(\F)-\j(\G)$, then $e_i>e_{i+1}$.  If $i\in \j(\G)-\j(\F)$, then $e_i<e_{i+1}$. Therefore,
\[
\j(\F) = \Des(\ee) \cup \{0,d\} \ \ \text{and} \ \  \j(\G) = \Asc(\ee) \cup \{d,n-k-1\}.
\]
\end{lemma}

\begin{proof}
It is clear from the definitions that $0 \in \j(\F)$ and $n-k-1 \in \j(\G)$.
By symmetry, it is enough to prove the first assertion. Assume for the sake of contradiction that $i\in \j(\F)-\j(\G)$ and $e_i<e_{i+1}$, so the table of $(\F|\G,\ee)$ contains
\[
\begin{array}{|ccccc|}
\hline
\cdots & F_i & \subsetneq & F_{i+1}& \cdots\\
\cdots & G_i & = & G_{i+1} & \cdots \\
\hline
\cdots &e_i & < & e_{i+1}& \cdots \\
\hline
\end{array}\, .
\]
Then the biflat $F_i|G_i$ arrives to the monomial $x_{\F|\G}$ after the biflat $F_{i+1}|G_{i+1}$ does, so $e_i \notin F_{i+1} \cap G_{i+1}$. This contradicts the fact that $e_i \in F_i \cap G_i \subseteq F_{i+1} \cap G_{i+1}$. 
\end{proof}

\begin{lemma}\label{lem:crossover}
If $i<j$ and $e_i<e_j$, then $e_i\not\in G_j$. If $i<j$ and $e_i>e_j$, then
$e_j\not\in F_i$.
\end{lemma}
\begin{proof}
It suffices to prove the first assertion. A part of the table of $(\F|\G , \ee)$ reads
\[
\begin{array}{|ccccc|}
\hline
F_i & \subseteq  & \cdots & \subseteq & F_j \\
G_i & \supseteq & \cdots & \supseteq & G_j\\
\hline
e_i &  & <  & & e_j
\\
\hline
\end{array} \ , 
\]
and hence $F_i|G_i$ appears  in the term $x_{\F|\G}$ after $F_j|G_j$ and $e_i \notin F_j \cap G_j$. Since $e_i \in F_i \subseteq F_j$, we must have $e_i \notin G_j$.
\end{proof}

\subsection{The resistant terms of $\delta^{n-k-1}$ are determined by the bottom row of their table.}

We now show that any resistant table $(\F|\G , \ee)$ is completely determined by the bottom row of its table, that is, by $\ee$, $d$, and $D$.  

\begin{proposition} \label{prop:bottomrow}
Let $(\F|\G,\ee)$, $D$, and $d$ be as above.  
\begin{enumerate}[(1)]\itemsep 5pt
\item\label{prop:bottomrow:1} 
For any $x \in D$, we have
\begin{align*}
B(x) &\coloneq \text{$\{e_{i+1} \mid i \in \j(\F) - d\} \cup x$  is a basis of  $\Tr_k \M$, and} \\
B^\perp(x) &\coloneq \text{$\{e_i \mid i \in \j(\G) - d\} \cup x$  is a basis of  $\M^\perp$}. 
\end{align*}
\item \label{prop:bottomrow:2} 
For any $x \in D$, the flags $\F$ and $\G$ are given by
\begin{align*}
F_j &= \begin{cases}
\cl(\{e_{i+1} \mid i \in \j(\F)-d, \, i+1 \leq j\}) & \textrm{ if } j \leq d, \\
\cl(x \cup \{e_{i+1} \mid i \in \j(\F)-d, \, i+1 \leq j\}) & \textrm{ if } j > d, 
\end{cases} \\
G_j &= \begin{cases}
\cl^\perp(x \cup \{e_i \mid i \in \j(\G)-d,\,  i \geq j\}) & \textrm{ if } j \leq d, \\
\cl^\perp(\{e_i \mid i \in \j(\G)-d, \, i \geq j\}) & \textrm{ if } j > d. 
\end{cases}
\end{align*}

\item \label{prop:bottomrow:3} 
The ranks of the flags $\F$ and $\G$ are given by
\[
\rank(F_j) = \abs{\j(\F)_{<j}} \ \ \text{and} \ \ 
\rank^\perp(G_j) = \abs{\j(\G)_{\geq j}}. 
\]
\end{enumerate}
\end{proposition}

The jump sets $\j(\F)$ and $\j(\G)$ are given by Lemma \ref{lem:ordering}, and hence $\F$ and $\G$ are determined by $\ee$, $d$, and $D$.  
Thus, the biflag $\F|\G$ is determined by the arrival sequence $\ee$, the double jump $d$, and the nonempty gap $D$.

\begin{proof}
Let us treat $B(x)$ as an ordered set, ordered from left to
right.  
If $i \in \j(\F) - \{d\}$ then $F_i \subsetneq F_{i+1}$ and $\rank(F_{i+1}) = \rank(F_i) + 1$ by the first part of Lemma \ref{lem:aeradicates}.
Also, Lemmas \ref{lem:ordering} and \ref{lem:crossover} tell us that $e_i > e_{i+1}$ and $e_{i+1} \notin F_i$. This implies that $F_{i+1} = \cl(F_i \cup e_{i+1})$, and that $e_{i+1}$ is independent from the earlier terms in $B(x)$. For $i=d$, since $x \in F_{d+1}-F_d$ by definition, we have that $F_{d+1} = \cl(F_d \cup x)$ and $x$ is independent from the earlier terms in $B(x)$. The same argument shows the analogous claims for $B^\perp(x)$. This proves the first and the second parts of the proposition.
The third part follows from 
\[
 \abs{\j(\F)_{<j}} = \begin{cases}
\abs{(\j(\F)-d)_{<j}} & \textrm{ if } j \leq d, \\
\abs{(\j(\F)-d)_{<j}} + 1 & \textrm{ if } j > d, \\
\end{cases}
\qquad
 \abs{\j(\G)_{\geq j}} = \begin{cases}
\abs{(\j(\F)-d)_{\geq j}} + 1 & \textrm{ if } j \leq d, \\
\abs{(\j(\F)-d)_{\geq j}} & \textrm{ if } j > d. \\
\end{cases}.
\]
This completes the proof.
\end{proof}

\subsection{The resistant terms of $\delta^{n-k-1}$ have no mixed biflats.}\label{ss:nomixedbiflats}

Proposition \ref{prop:bottomrow} tells us that, in order to describe the tables $(\F|\G,\ee)$ arising in the canonical expansion of $\delta^{n-k-1}$ that resist multiplication by $\gamma^k$, we may focus on the bottom row of their tables, that is, on $\ee$, $d$, and $D$. We now pursue this analysis further.
Call a biflat $F|G$ \emph{mixed} if both $F$ and $G$ are proper flats of $\M$ and $\M^\perp$, respectively.

\begin{proposition}\label{prop:nomixedbiflats} 
If the table $(\F|\G,\ee)$ arises in the canonical expansion of $\delta^{n-k-1}$ and the monomial $x_{\F|\G}$ resists multiplication by $\gamma^k$, then
\begin{enumerate}[(1)]\itemsep 5pt
\item
its unique double jump is at $d=r-k$, and 
\item
the resulting monomial $x_{\F|\G}$ has no mixed biflats, and its table is of the form
\[
\begin{array}{|cc|ccccccccccc|cc|}
\hline
\emptyset & \subsetneq &
F_1  & \subsetneq & \cdots & \subsetneq & F_{r-k} & \blue{\subsetneq} & E &  = & \cdots & =  & E
& = & E \\
E & = &
E & = & \cdots & = & E & \blue{\supsetneq} & G_{r-k+1} & \supsetneq & \cdots & \supsetneq  & G_{n-k-1}
& \supsetneq & \emptyset \\
\hline
&&
e_1  & > & \cdots &> & e_{r-k} & \blue{D}& e_{r-k+1}&<& \cdots &<& e_{n-k-1}
&& 
\\
\hline
\end{array}\, .
\]
\end{enumerate}
\end{proposition}

For the remainder of this subsection, we write $Y$ for the set
$ E - \{e_1, \ldots, e_{n-k-1}\} - D$ consisting of
 indices that does not appear in the bottom row of the table of $(\F|\G,\ee)$ when augmented with the entry $\blue{D}$.

\begin{example}\label{ex:resistantterm}
Before proving Proposition \ref{prop:nomixedbiflats}, let us illustrate it using Example \ref{ex:nbc biflag}. The graphical matroid of the cube shown in Figure \ref{fig:cube} has $n+1=12$ elements, rank $r+1=7$, and corank $r^\perp+1=5$. We saw in Examples \ref{ex:nbc biflag} and \ref{ex:nbc region} that, for $k=2$, one of the resistant tables $(\F|\G,\ee)$ in the canonical expansion of $\delta^{n-k-1} = \delta^8$ is the $\nbc$ monomial of the basis $B=\green{01}5\green{6}78b$ with $\IA(B) = \green{016}$, given by the table
\[
\begin{array}{|c|ccccccccc|c|}
\hline
\emptyset & b & 8b & 78b & 578b & \blue{\subsetneq} & E & E & E & E & E\\
E & E & E & E & E & \blue{\supsetneq} & 03469a & 469a & 69a & a & \emptyset\\
\hline
& b & 8 & 7 & 5 & \blue{12} & 3 & 4 & 9 & a & \\
\hline
\end{array}\, .
\]
The double jump occurs at $d=4$ and we have $D=\{1,2\}$ and $Y=\{0,6\}$. For either $x=1$ or $x=2$ 
the flats in $\F$ are the closures of the independent sets $b, 8b, 78b, 578b, x578b$ of $\M$, and the coflats in $\G$ are the coclosures of the independent sets $a, 9a, 49a, 349a, x349a$ of $\M^\perp$. 
\end{example}

\begin{proof} [Proof of Proposition \ref{prop:nomixedbiflats}]


We first show that $\min \j(\G) = d$. 
Suppose otherwise that $\G$ has jumps before $d$, and let $j-1<d$ be the position
at which the first one occurs.
Then $\set{0, 1, \ldots, j-2} \in \j(\F)$, and the table of $(\F|\G,\ee)$ reads
\[
\small  
\hspace{-.6cm}
\begin{array}{|ccccccccccccccccc|}
\hline
F_1  & \subsetneq & \cdots & \subsetneq & F_{j-1} & = &  F_j & \subseteq & \cdots & \subseteq & F_d & \blue{\subsetneq} & F_{d+1} &  \subseteq & \cdots & \subseteq  & F_{n-k-1}
\\
E & = & \cdots & = & E & \supsetneq & G_j & \supseteq & \cdots & \supseteq  & G_d & \blue{\supsetneq} & G_{d+1} & \supseteq & \cdots & \supsetneq  & G_{n-k-1} 
\\
\hline
e_1  &  & \cdots &  & e_{j-1} &  & e_j & & \cdots  & &   e_{d} & \blue{D}& e_{d+1}& & \cdots & & e_{n-k-1}
\\
\hline
\end{array}\, .
\]
Proposition \ref{prop:bottomrow} guarantees that $\set{e_1, \ldots, e_{j-1}}$ is independent and spans the flat $F_{j-1}=F_j$ of $\M$. Notice that $F_j \neq E$, since $j<d$. 
Now, Lemma \ref{lem:full flags} tells us that $\G$ is a complete flag of $\M^\perp$ with possible repetitions, so the coflat $G_j$ must be a hyperplane in $\M^\perp$, and hence  $C_j = E-G_j$ is a circuit of $\M$. But we have that $G_j \supseteq \{e_j, e_{j+1}, \ldots, e_{n-k-1}\} \cup D$, which implies that  $C_j = E-G_j  \subseteq \set{e_1, \ldots, e_{j-1}} \cup Y$. But $\set{e_1, \ldots, e_{j-1}}$ is independent in $\M$, so there must be an element $y \in Y$ such that $y \in C_j$. Then $y \notin G_j$, so $y \in F_j$. 

Recall from Proposition \ref{prop:bottomrow} that $B^\perp(x)\coloneq \set{e_i \mid i \in \j(\G) - d} \cup x$ is a basis of $\M^\perp$, so its complement
$\B^\perp(x)' \coloneq \set{e_i \mid i \in \j(\F)-0} \cup (D - x) \cup Y$  is a basis of $\M$,
and hence 
\[
\text{$I\coloneq \set{e_i \mid  i \in \j(\F)_{\leq d}-0} \cup y$ is independent in $\M$.} 
\]
Notice that $I$ is a subset of $F_d$ and 
$\abs{I} = \abs{\j(\F)_{\leq d}-0}+1  =  \abs{\j(\F)_{<d}}+1$;
this contradicts the third part of Proposition \ref{prop:bottomrow}, which says that  $\rank(F_d) = \abs{\j(\F)_{<d}}$.
We conclude that $\G$ has no jumps before $d$, that is, $G_1 = \cdots = G_d= E$.  


We next show that $\max \j(\F) = d$.
Suppose $\F$ has jumps after $d$, and let $j>d$ be the position at which the first such jump occurs.  
Since $\j(\F)\cup\j(\G)=\set{0,1,\ldots,n-k-1}$ and we previously showed that $0, 1, \ldots, d-1 \notin \j(\G)$, we must have $\j(\F)\supseteq\set{0,1,\ldots,d-1}$. Therefore, the table of $(\F|\G,\ee)$ reads
\[
\small 
\hspace{-.5cm}
\begin{array}{|ccccccccccccccccc|}
\hline
F_1  & \subsetneq & \cdots & \subsetneq & F_d & \blue{\subsetneq} &  F_{d+1} & = & \cdots & = & F_j & \subsetneq & F_{j+1} &  \subseteq & \cdots & \subsetneq  & F_{n-k-1}
\\
E & = & \cdots & = & E & \blue{\supsetneq} & G_{d+1} & \supsetneq & \cdots & \supsetneq  & G_j & = & G_{j+1} & \supseteq & \cdots & \supsetneq  & G_{n-k-1} 
\\
\hline
e_1  &  & \cdots &  & e_d &   \blue{D}  & e_{d+1} & & \cdots  & &   e_i & & e_{i+1}& & \cdots & & e_{n-k-1}
\\
\hline
\end{array}\, .
\]
This implies that the basis $B^\perp(x)'$ of $\M$ 
contains $J\coloneq \set{e_1, \ldots, e_d, e_j} \cup (D-x)$, which must then be independent.  But $J \subseteq F_j = F_{d+1}$,
and $\rank(F_{d+1}) = d+1$ because $\F$ is a complete flag of $\Tr_k(\M)$, with possible repetitions,  
by Lemma \ref{lem:full flags}, so
\[
d+1 \geq \rank(\abs{J}) = \abs{J} = d+1+\abs{D-x}.
\]
which implies $D=\set{x}$. It follows that $F_d \cup G_{d+1} = E-x$, contradicting Lemma \ref{lem:big unions}. 
It follows that $\F$ has no jumps after $d$, that is, $F_{d+1} = \cdots = F_{n-k-1} = E$. 
We conclude that $\j(\F)=\{0, 1, \ldots, d\}$ and $\j(\G)=\{d, d+1, \ldots, n-k\}$. The first part of Proposition \ref{prop:bottomrow}  then implies that $d = r-k$. 


The above discussion shows that $x_{\F|\G}$ has no mixed biflats. Furthermore, Lemma \ref{lem:ordering} tells us that 
$e_1 > \cdots > e_d$ and $e_{d+1} <  \cdots < e_{n-k-1}$. This completes the proof.
\end{proof}

%
%
We now strengthen Proposition \ref{prop:bottomrow} by showing that a resistant table $(\F|\G,\ee)$ in the canonical expansion of $\delta^{n-k-1}$ is completely determined by the arrival sequence $\ee$.

\begin{corollary}\label{cor:bottomrow2}
If $(\F|\G,\ee)$ be a table arising in the canonical expansion of $\delta^{n-k-1}$ such that  $x_{\F|\G}$  resists multiplication by $\gamma^k$, then 
\[
F_j | G_j = \begin{cases}
\cl\{e_1, \ldots, e_j\} | E & \textrm{ if } j \leq r-k, \\
E | \cl^\perp\{e_j, \ldots, e_{n-k+1}\}  & \textrm{ if } j > r-k. 
\end{cases} 
\]
In particular, the biflag $\F|\G$ is determined uniquely by the arrival sequence $\ee$.
\end{corollary}

\begin{proof}
This is a direct consequence of the second part of Proposition \ref{prop:bottomrow}, since $\j(\F)=\{0, 1, \ldots, r-k\}$ and $\j(\G)=\{r-k, r-k+1, \ldots, n-k\}$.
Note that $\ee$ also determines the nonempty gap $D = E-(F_{r-k} \cup G_{r-k+1})$.
\end{proof}

Although the resistant monomials in $\delta^{n-k-1}$ do not contain mixed biflats, note however that the multiplication by $\gamma^k$ may introduce mixed biflats in the canonical expansion of $\gamma^k\delta^{n-k-1}$, as we saw in Section \ref{sec:nbcmonomials}.


\subsection{The resistant terms of $\delta^{n-k-1}$ are $\nbc$ monomials.}

In order to identify the resistant terms in the canonical expansion of $\delta^{n-k-1}$, we need to recall a few fundamental facts
from the theory of basis activities of a matroid $\M$ on a 
linearly ordered ground set, as
developed by Tutte \cite{Tutte67} and Crapo \cite{Cr69}.

\begin{definition}\label{def:IA(S)}
For a subset $S \subseteq E$ of the ground set of  $\M$, we set
\begin{align*}
\IA(S) &\coloneq  \set{e\in S\mid \text{there exists a cocircuit $C^\perp \subseteq (E-S) \cup e$ 
for which $e=\min C^\perp$}}, \\
\EA(S) &\coloneq  \set{e\in E-S\mid \text{there exists a circuit $C \subseteq S \cup e$ 
for which $e=\min C$}}. 
\end{align*}
\end{definition}

The following set will play a very important role.

\begin{propdef}\label{def:P_M(S)}
For  an independent set $S \subseteq E$ of  $\M$, let
\begin{align*}
\P(S) &\coloneq  \set{e\in E-S\mid \text{there exists a cocircuit $C^\perp \subseteq E-S$ 
for which $e=\min C^\perp$}} \\
&= \set{e\in E-S\mid \text{there exists a cocircuit $C^\perp \subseteq E-\cl(S)$ 
for which $e=\min C^\perp$}} \\
&= \text{lexicographically smallest set such that $S \sqcup \P(S)$ is a basis}.
\end{align*}
\end{propdef}



\begin{proof}
The equivalence of the first and third definitions is shown in \cite{Cr69}. To show the equivalence of the first two, suppose $C^\perp \subseteq E-S$ is a cocircuit with $e=\min C^\perp$. 
If we had an element $f\in C^\perp\cap \cl(S)$, there would be a circuit $C\subseteq S\cup f$
containing $f$.  But then we would have $C\cap C^\perp=\set{f}$, which is impossible by Lemma \ref{lem:big unions}, so
$C^\perp\subseteq E-\cl(S)$. This proves one inclusion, and the reverse inclusion is trivial.
\end{proof}

The next proposition shows that there is a close relationship between $S$ and the lexicographically smallest basis $B$ containing it \cite[Section 2]{LasVergnas13}.

\begin{proposition}\label{prop:basis_decomp}
The following holds for any independent subset $S\subseteq E$ of  $\M$.
\begin{enumerate}[(1)]\itemsep 5pt
\item $B\coloneq S\cup \P(S)$ is the lexicographically smallest basis of $\M$ containing $S$.
\label{prop:basis(1)}
\item $\IA(B)=\IA(S)\cup \P(S)$.
\label{prop:basis(2)}
\item $\EA(B)=\EA(S)$.
\end{enumerate}
\end{proposition}

In the upcoming arguments, the reader may find it useful to consult Figure \ref{fig:activities}, which summarizes Proposition \ref{prop:basis_decomp}.

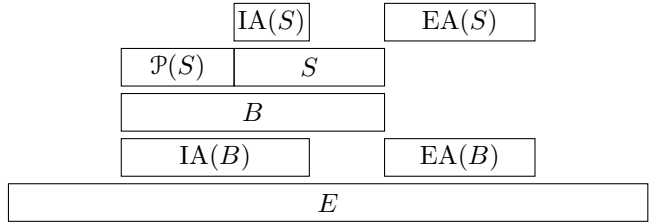
\begin{figure}[h]
  \centering
  \begin{tikzpicture}
    \def\dy{0.6}
    \draw (0.5,0) rectangle node{$E$} (9,0.5);
    \draw (2,\dy) rectangle node{$\IA(B)$} (4.5,0.5+\dy);
    \draw (5.5,\dy) rectangle node{$\EA(B)$} (7.5,0.5+\dy);
    \draw (2,2*\dy) rectangle node{$B$} (5.5,0.5+2*\dy);  
    \draw (2,3*\dy) rectangle node{$\P(S)$} (3.5,0.5+3*\dy);
    \draw (3.5,3*\dy) rectangle node{$S$} (5.5,0.5+3*\dy);    
    \draw (3.5,4*\dy) rectangle node{$\IA(S)$} (4.5,0.5+4*\dy);
    \draw (5.5,4*\dy) rectangle node{$\EA(S)$} (7.5,0.5+4*\dy);
  \end{tikzpicture}
%
\caption{Activities for an independent $S$ and its greedy completion to a basis $B$.\label{fig:activities}}
\end{figure}

\begin{example}\label{ex:basis_decomp}
  For the graphical matroid of the pyramid in Figure~\ref{fig:pyramid} and
  $S=15$, we have $\P(S)=23$ and $B=01234$.  Then
$\IA(B) = 23, \IA(S)=\emptyset,$ and $
\EA(B) = \EA(S) = 0.$
\end{example}

We will make use of the following fact.

\begin{lemma}\label{lem:big unions II}
Let $S$ be an independent set of $\M$, and let $F=\cl(S)$.  If $G$
is a hyperplane of $\M^\perp$ satisfying $F\cup G\neq E$, then 
$F\cup G\cup \P(S)\neq E$ as well.
\end{lemma}

\begin{proof}
Proposition~\ref{prop:basis_decomp} tells us that $B\coloneq S\cup \P(S)$ is a basis of $\M$ 
and $\P(S) \subseteq \IA(B)$. 
Thus, for each element $p\in \P(S)$, we can find a unique cocircuit $C^\perp(p) \subseteq (E-B) \cup p$ of $\M$ for which $p=\min C^\perp_p$. Notice that $C^\perp(p)\cap \P(S)=\set{p}$. 
Set $H(p)\coloneq E-C^\perp(p)$, a hyperplane of $\M$.

We claim $F=\bigcap_{p \in \P(S)}H(p)$.  Since each $C^\perp(p)$ is contained in $E-S$ we have $S\subseteq \bigcap_{p\in \P(S)}H(p)$.
Since the intersection of flats is a flat, $F = \cl(S) \subseteq \bigcap_{p \in \P(S)} H(p)$ as well.
On the other hand, each $H(p)$ is a hyperplane, so the submodular inequality
gives
\[
\rank \big( \bigcap_{p\in \P(S)}H(p)\big)  \leq \rank(\M) -\abs{\P(S)} = \abs{S} = \rank(F)
\]
which proves the claim.

Now, the complement $C\coloneq E-G$ is a circuit of $\M$.  Let $D=E-(F\cup G) \neq \emptyset$.  Then
\[
D =(E-F)\cap (E-G)
= \big(\bigcup_{p\in \P(S)}C^\perp(p)\big)\cap C
= \bigcup _{p\in \P(S)}(C^\perp(p)\cap C).
\]
Since $D$ is nonempty, $C^\perp(p) \cap C\neq\emptyset$ for some $p$.  Since $C^\perp(p)$
is a cocircuit and $C$ is a circuit, $C^\perp(p) \cap C$ -- and hence $D$ -- contains at least
one element $q \neq p$.  Since ${C^\perp(p) \cap \P(S)}=\set{p}$, it follows that $q \notin \P(S)$ so 
$q \notin F \cup G \cup \P(S)$, as desired.
\end{proof}

%

In Proposition \ref{prop:nomixedbiflats}, we showed that the terms of the canonical expansion of $\delta^{n-k-1}$ that resist multiplication by $\gamma^k$ are given by tables $(\F|\G,\ee)$ of the form 
\[
\begin{array}{|cc|ccccccccccc|cc|}
\hline
\emptyset & \subsetneq &
F_1  & \subsetneq & \cdots & \subsetneq & F_{r-k} & \blue{\subsetneq} & E &  = & \cdots & =  & E
& = & E \\
E & = &
E & = & \cdots & = & E & \blue{\supsetneq} & G_{r-k+1} & \supsetneq & \cdots & \supsetneq  & G_{n-k-1}
& \supsetneq & \emptyset \\
\hline
&&
e_1  & > & \cdots &> & e_{r-k} & \blue{D}& e_{r-k+1}&<& \cdots &<& e_{n-k-1}
&& 
\\
\hline
\end{array} \, . 
\]
The following proposition describes precisely which tables arise.

\begin{proposition}\label{prop:delta expansion II}
Let
 $S$ be the independent set $\set{e_1 > \cdots > e_{r-k}}$ of $\M$, and let $B$ be the basis $B = S \sqcup \P(S)$ of $\M$ in Proposition \ref{prop:basis_decomp}.
\begin{enumerate}[(1)]\itemsep 5pt
\item \label{enum:delta 0}
$\IA(S) = \EA(S) = \emptyset$.
\item 
$\EA(B) = \emptyset$, so $B$ is a $\nbc$ basis of $\M$. \label{enum:delta 1}
\item
$\IA(B)=\P(S)$ and $|\IA(B)|=k+1$.\label{enum:delta 2}
\item $e_i=\min F_i$ for $1\leq i\leq r-k$. \label{enum:delta 3}
\item $e_i=\min (G_i-\P(S))$ for $r-k+1\leq i\leq n-k-1$. \label{enum:delta 4}
\item
$\F|\G = {\F(B)| \G(B)}$
and $x_{\F|\G}$ is the $\nbc$ monomial of the $\nbc$ basis $B$, as  in Proposition \ref{def:nbc biflags}. \label{enum:delta 5}
\end{enumerate}
\end{proposition}

\begin{example} \label{ex:resistant monomial}
For the cube graph of Figure \ref{fig:cube} and $k=2$, let us revisit the resistant monomial
\[
x_{b|E} \, x_{8b|E} \, x_{78b|E} \, x_{578b|E} \, x_{E|03469a} \, x_{E|469a} \, x_{E|69a} \, x_{E|9}.
\]
The arrival sequence $\ee = (b, 8, 7, 5; 3, 4, 9, a)$ was computed in Example \ref{ex:resistantterm}. We have $S=b875$, and its lexicographically smallest completion to a basis is given by $\P(S) = 016$. Then $B=\green{01}5\green{6}78b$ is indeed a $\nbc$ basis with $\IA(B) = \green{016}$ . Conditions (3) and (4) are easily checked directly, and  ${\F|\G} = {\F(B)| \G(B)}$ as described in Example \ref{ex:nbc biflag}.
\end{example}

We prepare the proof of  Proposition \ref{prop:delta expansion II} with  some technical lemmas.

\begin{lemma}\label{lem:unmixed expansion}
Let  $(\F|\G,\ee)$ be as in Proposition \ref{prop:delta expansion II}.  
\begin{enumerate}[(1)]\itemsep 5pt
\item
For each $1\leq i\leq r-k$, there is an index $r-k+1 \leq j \leq n-k$ such that
\[
e_i=\max\big(E-(F_{i-1}\cup G_j)\big).
\]
It is  the smallest index $j \geq r-k+1$ such that $e_j>e_i$, or $j=n-k$ if there is no such index.
\item
For each $r-k+1\leq j\leq n-k-1$, there is an index  $0 \leq i\leq r-k$ such that
\[
e_j=\max\big(E-(F_i\cup G_{j+1})\big).
\]
It is  the largest index $i \leq r-k$ such that $e_i>e_j$, or $i=0$ if there is no such index.
\end{enumerate}
\end{lemma}

\begin{proof}
This is a straightforward restatement of Proposition \ref{prop:expansion}, since we have determined the table for $(\F|\G,\ee)$ in Proposition \ref{prop:nomixedbiflats}. 
\end{proof}

\begin{lemma}\label{lem:D is small}
Every element of the nonempty gap $D$ is smaller than every $e_i$.
\end{lemma}
\begin{proof}
The smallest $e_i$ is either $e_{r-k}$ or $e_{r-k+1}$. 
If $e_{r-k}<e_{r-k+1}$, then Lemma~\ref{lem:unmixed expansion} implies that
\[
e_{r-k} = \max\big(E-(F_{r-k-1}\cup G_{r-k+1})\big) \geq \max D,
\]
where we used $F_{r-k-1}\cup G_{r-k+1} \subseteq F_{r-k}\cup G_{r-k+1} = E-D$. Since $D\cap \set{e_1,\ldots,e_{n-k-1}}$ is empty, the inequality is strict.
If $e_{r-k}>e_{r-k+1}$, a similar argument shows $e_{r-k+1} > \max D$.  
\end{proof}

The first part of Proposition \ref{prop:bottomrow} tells us that $S$ is independent in $\M$,
and the first part of Proposition \ref{prop:basis_decomp} tells us that
the set $B\coloneq S\sqcup \P(S)$ is a basis. 
In particular, $\abs{\P(S)}=k+1$ and $\Q(S) = \emptyset$.  
Our next result relates $\P(S)$ with the partition $E = \set{e_1,\ldots,e_{n-k-1}} \sqcup D\sqcup Y$. 
We illustrate  this in Figure \ref{fig:partitionDY}, which is a refinement of Figure \ref{fig:activities} in the case
$\IA(S) = \EA(S) = \emptyset$.

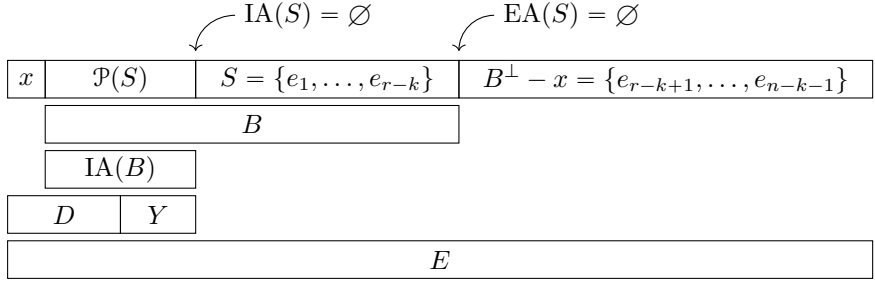
\begin{figure}[h]
  \centering
  \begin{tikzpicture}
    \def\dy{0.6}
    \draw (1,0) rectangle node{$E$} (12.5,0.5);
    \draw (1,\dy) rectangle node{$D$} (2.5,0.5+\dy);
    \draw (2.5,\dy) rectangle node{$Y$} (3.5,0.5+\dy);
    \draw (1.5,2*\dy) rectangle node{$\IA(B)$} (3.5,0.5+2*\dy);
    \draw (1.5,3*\dy) rectangle node{$B$} (7,0.5+3*\dy);
    \draw (1,4*\dy) rectangle node{$x$} (1.5,0.5+4*\dy);
    \draw (1.5,4*\dy) rectangle node{$\P(S)$} (3.5,0.5+4*\dy);
    \draw (3.5,4*\dy) rectangle node{$S=\set{e_1,\ldots,e_{r-k}}$}
    (7,0.5+4*\dy);
    \draw (7,4*\dy) rectangle
    node{$B^\perp-x=\set{e_{r-k+1},\ldots,e_{n-k-1}}$} (12.5,0.5+4*\dy);
    \node (one) at (3.5,0.5+4*\dy) {};
    \node (two) at (7,0.5+4*\dy) {};
    \node (three) at (5,0.5+5*\dy) {$\IA(S)=\emptyset$};
    \node (four) at (8.5,0.5+5*\dy) {$\EA(S)=\emptyset$};
    \draw[<-,out=90,in=180] (one) to (three);
    \draw[<-,out=90,in=180] (two) to (four);
  \end{tikzpicture}

\caption{Activities and the partition of $E$. 
\label{fig:partitionDY}}
\end{figure}

\begin{lemma}\label{lem:DY vs DS}
For $x = \min(E-B) \in D$, we have 
\[
E-\set{e_1,\ldots,e_{n-k-1}}=D\sqcup Y = \P(S)\sqcup \set{x}.
\]
\end{lemma}

\begin{proof}
We first show that $\P(S)\subseteq D\sqcup Y$. By way of a contradiction, suppose $e_i\in 
\P(S)$ for some $i$. Since $e_1, \ldots, e_{r-k} \in S$, we must have $i\geq r-k+1$.
By Definition~\ref{def:P_M(S)}, there is a cocircuit $C^\perp$ for which $e_i=\min C^\perp$ and $C^\perp \subseteq E - \cl(S) = E - F_{r-k}$. Also, 
Lemma~\ref{lem:unmixed expansion} tells us that
\[
e_i =\max
(E-F_j)  \cap (E-G_{i+1})
 \ \   \text{for some $j\leq r-k$}.
\]
However, since $e_i \notin G_{i+1}$ by Lemma \ref{lem:crossover}, we have 
\[
(E-F_j) \cap (E-G_{i+1}) \supseteq (E-F_{r-k}) \cap (E-G_{i+1}) \supseteq C^\perp \cap (E-G_{i+1}) \ni e_i.
\]
Therefore, we must also have
$e_i =\max\big((E-G_{i+1}) \cap C^\perp\big)$.
Since $e_i = \min C^\perp$ we must have $(E-G_{i+1}) \cap C^\perp = \set{e_i}$, and hence $\abs{G_{i+1} \cup (E-C^\perp)} = n$; this contradicts Lemma \ref{lem:big unions}.

Next we show that there must be an element $x \in D$ that is not in $\P(S)$. 
To do that, we invoke Lemma \ref{lem:big unions II}.  Since $\cl(S) \cup G_{r-k+1} = F_{r-k} \cup G_{r-k+1} \neq E$, we must have some element $x \notin F_{r-k} \cup G_{r-k+1} \cup \P(S)$. However, $x \notin F_{r-k} \cup G_{r-k+1}$ means that $x \in D$.

We observe that $\abs{D \sqcup Y} = \abs{E} - \abs{\set{e_1, \ldots, e_{n-k-1}}} = (n+1) - (n-k-1) = k+2$, while $\abs{\P(S)}=\abs{B}-\abs{S} = (r+1)-(r-k) = k+1$.  Since our
$x\not\in \P(S)$, the inclusion $\P(S)\sqcup\set{x} \subseteq D \sqcup Y$ must in fact be an equality.

Finally, Lemma \ref{lem:D is small} and $E-B = \{x, e_{r-k+1}, \ldots, e_{n-k}\}$ give $x = \min(E-B)$.
\end{proof}

\begin{lemma}\label{lem:Y is in top G}
We have the inclusion $Y\subseteq G_{r-k+1}$.
\end{lemma}
\begin{proof}
We have that $Y\subseteq F_{r-k}\cup G_{r-k+1} = E-D$. If the claim were not true, there would be an element $y \in Y$ such that $y \in F_{r-k} = \cl(S)$, so $S \cup y$ would be dependent. But Lemma~\ref{lem:DY vs DS} would then imply that $y \in Y \subseteq \P(S)$, contradicting the fact that $S \cup P(S)$ is a basis.
\end{proof}

We now  prove our description of the resistant terms in the canonical expansion of $\delta^{n-k-1}$. 

\begin{proof}[Proof of Proposition~\ref{prop:delta expansion II}] We prove parts (4), (5), (1), (2), (3), (6), in that order.
\begin{enumerate}[(1)]\itemsep 5pt
    \item[(4)]
Let $1\leq i\leq r-k$. Lemma~\ref{lem:unmixed expansion} provides a $j\geq r-k+1$ for which $e_i=\max\big(E-(F_{i-1}\cup G_j)\big)$.   We claim that
\[
F_i\cup G_j \subseteq \set{e_1,\ldots,e_i} \sqcup \set{e_j,e_{j+1},\ldots,e_{n-k-1}} \sqcup D \sqcup Y.
\]
In view of the decomposition $E=\set{e_1,\ldots,e_{n-k-1}}\sqcup D\sqcup Y$, this amounts to checking that $e_h\not\in F_i\cup G_j$
for $i<h<j$, which follows from Lemma~\ref{lem:crossover}.

Now assume, for the sake of contradiction, that $\min F_i = a < e_i$. Since $e_1>\cdots>e_i$ and $e_{n-k-1} > \cdots > e_j > e_i$, we must have $a \in D \sqcup Y = \P(S)\sqcup x$, by Lemma~\ref{lem:DY vs DS}.  
However, $D \cap F_i = \emptyset$ and $a \in F_i$, so $a$ cannot be $x$.  Additionally, $\P(S)$ is independent from $S$ and $a \in F_i\subseteq \cl(S)$, so $a$ cannot be in $\P(S)$ either. 
We conclude $e_i=\min F_i$ for $i\leq r-k$.

\item[(5)] Let $r-k+1 \leq j \leq n-k-1$. As before, Lemma~\ref{lem:unmixed expansion} provides a $0 \leq i\leq r-k$ for which
\[
F_i\cup G_j \subseteq \set{e_1,\ldots,e_i} \sqcup \set{e_j,e_{j+1},\ldots,e_{n-k-1}} \sqcup D \sqcup Y.
\]
Assume, for the sake of contradiction, that $\min \big(G_j-\P(S)\big) = a < e_j$. Since $e_1>\cdots>e_i>e_j$ and $e_{n-k-1} > \cdots > e_j$, we must have $a \in D \cup Y$. But $a \notin \P(S)$ by definition, and $D \cap G_j = \emptyset$ so $a \neq x$.  The desired result follows.


\item[(1)] If we had $a\in\EA(S)$, then $S\cup a $ would contain a circuit $C$ with $a=\min C$. For the largest $i$ with $e_i\in C$, we would then have $C-a \subseteq F_i$. Since $C$ is a circuit and $F_i$ is a flat, this would imply that $a\in F_i$, contradicting that $e_i=\min F_i$ as shown in 4. 
Thus $\EA(S)=\emptyset$.

Suppose we had $e_i \in \IA(S)$ for some $1 \leq i \leq r-k$. Then $e_i\in \IA(B)$, so there is a cocircuit $C^\perp \subseteq (E-B) \cup e_i = \{e_{r-k+1}, \ldots, e_{n-k}, x, e_i\}$ with $e_i=\min C^\perp$. By Lemma~\ref{lem:D is small}, this means that $x \notin C^\perp$. Therefore 
$C^\perp-e_i \subseteq G_{r-k+1}$. But then, since $C^\perp$ is a circuit and $G_{r-k+1}$ is a flat in $\M^\perp$, we must have $e_i \in G_{r-k+1}$ as well; this contradicts Lemma~\ref{lem:ordering} since $i<r-k+1$ and $e_i<e_{r-k+1}$.

\item[(2)]
  Proposition \ref{prop:basis_decomp} and 1. tell us that $B$ is a basis and $\EA(B)= \EA(S)=\emptyset$.
\item[(3)]
  Proposition \ref{prop:basis_decomp} and 2 tell us that $\IA(B)=\P(S)$.
\item[(6)]
  By (2) and (3), we have 
\[
B - \IA(B) = S = \{e_1 > \cdots > e_{r-k}\}
\]
and Lemma \ref{lem:DY vs DS} tells us that 
\[
E-B - \min(E-B) = \{e_{r-k+1} < \cdots < e_{n-k-1}\}
\]
Therefore, by Corollary \ref{cor:bottomrow2}, the flags $\F$ and $\G$ 
are precisely the flags $\F(B)$ and $\G(B)$ of the $\nbc$ biflag of $B$, as described in Proposition \ref{def:nbc biflags}. \qedhere
\end{enumerate}
\end{proof}

\begin{corollary}\label{cor:no multiplicity}
Let $(\F|\G,\ee)$ be a table arising in the canonical expansion of $\delta^{n-k-1}$ such that  $x_{\F|\G}$  resists multiplication by $\gamma^k$. Then the arrival sequence $\ee$ is determined uniquely by the biflag $\F|\G$.
\end{corollary}

\begin{proof}
The arrival sequence $\ee$ is determined by (3) and (4) of Proposition \ref{prop:delta expansion II}.
\end{proof}


We are finally ready to prove our description of the canonical expansion of $\gamma^k\delta^{n-k-1}$ into monomials.


\begin{proof}[Proof of Theorem \ref{thm:main}]
Every resistant table $(\F|\G,\ee)$ in the canonical expansion of $\delta^{n-k-1}$ gives a $\nbc$ monomial $x_{\F(B)| \G(B)}$ by the last part of Proposition \ref{prop:delta expansion II}. Every such monomial does appear in this expansion by Proposition \ref{prop:delta expansion III}. Furthermore, it appears only once by Corollary \ref{cor:no multiplicity}. Therefore, we have 
\[
\gamma^k \delta^{n-k-1} = \sum_{\substack{B \nbc \textrm{basis} \\ |\IA(B)| = k+1}} \gamma^k x_{\F(B)| \G(B)}.
\]
The desired formula for $\gamma^k \delta^{n-k-1}$ then follows by Proposition \ref{prop:nbc is resistant}.
\end{proof}

\bibliographystyle{amsalpha}
\bibliography{refs}

\end{document}